\documentclass[12pt]{article}
\usepackage{amsmath,amssymb,amsthm}

\usepackage{enumitem}  

\usepackage{breqn}
\usepackage{setspace}

\usepackage{graphicx}	        
\usepackage{tikz-cd} 
\usepackage{bm}
\usepackage{xparse}
\usepackage{float}

\newtheorem{theorem}{Theorem}[section]
\newtheorem{proposition}[theorem]{Proposition}
\newtheorem{lemma}[theorem]{Lemma}
\newtheorem{definition}[theorem]{Definition}

\newtheorem{corollary}[theorem]{Corollary}
\newtheorem{remark}[theorem]{Remark}

\newtheorem*{question}{Question}
\newtheorem*{notation}{Notation}
\allowdisplaybreaks[1]

\NewDocumentCommand{\evalat}{sO{\big}mm}{%
  \IfBooleanTF{#1}
   {\mleft. #3 \mright|_{#4}}
   {#3#2|_{#4}}%
}

\makeatletter
\newcommand{\colim@}[2]{%
  \vtop{\m@th\ialign{##\cr
    \hfil$#1\operator@font colim$\hfil\cr
    \noalign{\nointerlineskip\kern1.5\ex@}#2\cr
    \noalign{\nointerlineskip\kern-\ex@}\cr}}%
}
\newcommand{\colim}{%
  \mathop{\mathpalette\colim@{\rightarrowfill@\textstyle}}\nmlimits@
}
\makeatother

\makeatletter
\newcommand*{\rom}[1]{\expandafter\@slowromancap\romannumeral #1@}
\makeatother

\begin{document}
\title{The Existence of Infinitely Many Geometrically Distinct Non-constant Prime Closed Geodesics on Riemannian Manifolds}

\author{Sergio Charles}

\date{10 August 2018}

\maketitle

\begin{abstract}
We enumerate a necessary condition for the existence of infinitely many geometrically distinct, non-constant, prime closed geodesics on an arbitrary closed Riemannian manifold $M$. That is, we show that any Riemannian metric on $M$ admits infinitely many prime closed geodesics such that the energy functional $E:\Lambda M\to\mathbb{R}$ has infinitely many non-degenerate critical points on the free loop space $\Lambda M$ of Sobolev class $H^1=W^{1,2}$. This result is obtained by invoking a handle decomposition of free loop space and using methods of cellular homology to study its topological invariants.
\end{abstract}

\section{Introduction and Preliminaries\label{sec:rathomotopy}}
Let $M$ be a multiply connected Riemannian manifold. Then the conjugacy classes in the fundamental group $\pi_1(M)$ may be mapped bijectively onto the free homotopy classes of closed loops on $M$. Closed geodesics are generated by undergoing an energy minimizing procedure. Thus, every closed multiply connected Riemannian manifold $M$ carries a closed geodesic. Likewise, every closed simply connected Riemannian manifold $M$ carries a closed geodesic according to Lyusternik and Fet \cite{lyusternik-fet}. 
The existence of at least one non-trivial, non-constant, closed geodesic on a closed multiply connected Riemannian manifold follows from properties of the gradient flow of the energy functional. 
\begin{theorem}\label{Cartan} Assume $M$ is closed and $\pi_1(M)$ is non-trivial. There is a closed geodesic in each non-trivial free homotopy class or in each non-trivial conjugacy class of the fundamental group $\pi_1(M)$.
\end{theorem}
The aforementioned theorem is only valid for $\pi_1(M)$ non-trivial, i.e. when $M$ is multiply connected. The theorem of Lyusternik--Fet covers the case for $M$ simply connected.
\begin{theorem}[Lyusternik--Fet \cite{lyusternik-fet}]
Any closed simply connected Riemannian manifold carries at least one closed geodesic.
\end{theorem}

\begin{definition}
A non-constant closed geodesic is said to be prime if it traces out its image exactly once and if it is not the iterate of another closed geodesic. Similarly, two geodesics are geometrically distinct if their images differ as subsets of $M$.
\end{definition} 
Two geodesics have the same image if and only if their parameterizations differ modulo an affine transformation of $\mathbb{R}$. Note, in general, prime closed geodesics are not simple; that is, they admit self-intersections. We henceforth address the following question.
\begin{question}
Does every closed Riemannian manifold $M$ admit infinitely many geometrically distinct, non-constant, prime closed geodesics?
\end{question}
The case for $S^2$ has been proven in the affirmative by Hingston \cite{hingston}. However, the solution is not known in general. In the case of multiply connected manifolds, we can enumerate a class for which the fundamental group $\pi_1(M)$ has infinitely many closed geodesics, whereby the minimization procedure produces infinitely many geometrically distinct geodesics, one for each homotopy class. As such, the essence of the problem is reduced to the simply connected case, which was addressed by Gromoll and Meyer \cite{gromoll-meyer}. 

\begin{notation}
We will use the following conventional notations.
\begin{enumerate}[label=(\roman*)]
\item $\Omega M$ denotes the loop space of a manifold $M$.
\item $\mathcal{L} M$ denotes the smooth free loop space on a manifold $M$.
\item $\Lambda M$ denotes the free loop space of Sobolev class $H^1=W^{1,2}$ on a manifold $M$.
\item $\mathbb{F}_p$ denotes the field with $p$ elements for $p\ge 2$ prime.
\item $b_k(\Lambda M;\mathbb{F}_p):=\text{rank }H_k(\Lambda M;\mathbb{F}_p)$ is the $k$-th Betti number of $\Lambda M$ in $\mathbb{F}_p$ coefficients, which is equivalent to the $k$-th Betti number of $\mathcal{L} M$ in $\mathbb{F}_p$ coefficients.
\end{enumerate}
\end{notation}

\begin{definition}
Closed geodesics $\gamma:S^1=\mathbb{R}/\mathbb{Z}\to M$ are critical points of the energy functional 
\begin{equation}\label{energyfunctional}
E[\gamma]=\int_{S^1}\|\dot\gamma(t)\|^2dt,
\end{equation} which is defined on the free loop space $\Lambda M=H^1(S^1,M)$ of Sobolev class $H^1$.
\end{definition}
Thus, the critical points of $E$ are determined by the topology of the Hilbert manifold $\Lambda M$. Let the bilinear map $\nabla:\Gamma(TM)\times\Gamma(TM)\to\Gamma(TM)$ on sections of the tangent bundle $TM$ be the Levi-Civita connection on $M$. For a smooth curve $\gamma: I\subset\mathbb{R}\to M$ parameterized by $t$, the covariant derivative $\nabla_{\dot\gamma}\dot\gamma :=\nabla_{t}\dot\gamma$ defines a vector field along $\gamma$. The curve $\gamma$ is said to be a geodesic if it satisfies the autoparallel transport equation 
\begin{equation}\label{autoparallel}
\nabla_t\dot\gamma=0.
\end{equation} 
A geodesic has constant speed $\|\dot\gamma(t)\|$ because $\frac{1}{2}\frac{d}{dt}\|\dot\gamma\|^2=\langle \nabla_t\dot\gamma,\dot\gamma\rangle$. Equation (\ref{autoparallel}) is a second-order, non-linear, ordinary differential equation with smooth coefficients. 

Let $p,q\in M$ be distinct points, and consider the space of smooth paths defined on $[0,1]$ from $p$ to $q$:
\begin{equation}
\mathcal{P}(p,q)=\{\gamma\in C^{\infty}([0,1],M):\gamma(0)=p,\gamma(1)=q\}.
\end{equation}
The space of such paths $\mathcal{P}(p,q)$ is a Fr\'echet manifold and the tangent space at a path $\gamma$ is defined as $T_{\gamma}P(p,q)=\{\xi\in\Gamma(\gamma^*TM):\xi(0)=0,\xi(1)=0\}$. That is, an element $\xi\in T_{\gamma}\mathcal{P}(p,q)$ uniquely determines a curve 
\begin{equation}
\alpha_{\xi}:\mathcal{V}(0)\subset\mathbb{R}\to\mathcal{P}(p,q),\quad \alpha_{\xi}(s)(t):=exp_{\gamma(t)}(s\xi(t)),
\end{equation} where $\alpha_{\xi}(0)=\gamma$ and $\evalat[\big]{\frac{d}{ds} }{s=0}\alpha_{\xi}(s)(t)=\xi(t),t\in[0,1]$. The energy functional $E:\mathcal{P}(p,q)\to\mathbb{R}$ for this Fr\'echet manifold is defined as 
\begin{equation}\label{energyfunctional}
E[\gamma]:=\int_0^1 \|\dot\gamma(t)\|^2dt.
\end{equation} Similarly, the differential of $E$ at $\gamma$ is a linear map $dE[\gamma]:T_\gamma\mathcal{P}(p,q)\xrightarrow{\sim}\mathbb{R}$ given by the first variation formula via its action on a smooth section of the tangent bundle $\xi\in T_{\gamma}\mathcal{P}(p,q)$, 
\begin{equation}
dE[\gamma]\xi=2\int_0^1\langle\nabla_t\xi,\dot\gamma\rangle=-2\int_0^1\langle\xi,\nabla_t\dot\gamma\rangle.
\end{equation}
\begin{remark}
The curve $\gamma\in C^{\infty}([0,1],M)$ is a critical point of $E$ if and only if $\nabla_t\dot\gamma=0$; that is, $\gamma$ is a geodesic from $p$ to $q$.
\end{remark} 
For our purposes, we examine Banach and, more specifically, Hilbert manifolds. Consider the following space of paths of Sobolev class $H^1$ defined on $[0,1]$ from $p$ to $q$ as the domain for the energy functional:
\begin{equation}
\Omega(p,q)=\{\gamma\in H^1([0,1],M):\gamma(0)=p,\gamma(1)=q\}.
\end{equation}
Such paths are continuous so the endpoint conditions are well-defined. The space $\Omega(p,q)$ is, in fact, a smooth Hilbert manifold with tangent space at $\gamma\in \Omega(p,q)$ given by the vector space of vector fields along $\gamma$ belonging to the Sobolev class $H^1$ and vanishing at endpoints:
\begin{equation}
T_{\gamma}\Omega(p,q)=\{\xi\in H^1(\gamma^*TM):\xi(0)=0,\xi(1)=0\}.
\end{equation}
Moreover, the energy functional $E:\Omega(p,q)\to\mathbb{R}$ is $C^2$. The critical points are the $H^1$-solutions of the elliptic autoparallel transport equation $\nabla_t\dot\gamma=0$. As such, the geodesic solutions are necessarily smooth. 

Consider the length functional 
\begin{equation}
L:\mathcal{P}(p,q)\to\mathbb{R}, \quad L[\gamma]:=\int_0^1\|\dot\gamma(t)\|^2dt.
\end{equation} 
The functional $L$ is invariant under the action of the infinite-dimensional group of diffeomorphisms on the interval [0,1], i.e. $L[\gamma(gt)]=L[\gamma(t)]$ for all $g\in G:=\text{Diffeo}([0,1])$. Thus, the critical points of $L$ form infinite-dimensional families of solutions. The degeneracy is rectified by observing that some path $\gamma\in\mathcal{P}(p,q)$ admits positive reparameterizations on $[0,1]$ with constant speed \cite{alexandru}. A path $\gamma\in\mathcal{P}(p,q)$ is a geodesic if and only if it is a critical point of the length functional $L[\gamma]$ and has constant speed.

Consider the quotient space $S^1=\mathbb{R}/\mathbb{Z}$ and the space of smooth loops in $M$: 
\begin{equation}
\mathcal{L}M:=C^{\infty}(S^1,M),
\end{equation} which is a Fr\'echet manifold on which the $2$-dimensional orthogonal group $O(2)=SO(2)\rtimes\{\pm 1\}$ acts by $(e^{i\theta}\cdot\gamma)(t):=\gamma(t+\theta)$ and $(-1\cdot\gamma)(t):=\gamma(-t)$. For economy, we consider the smooth Banach manifold
\begin{equation}
\Lambda M:=H^1(S^1,M),
\end{equation} which is the so-called $\textit{free loop space}$ of $M$. The inclusion map $\mathcal{L}M\hookrightarrow \Lambda M$ is a homotopy equivalence, and the $O(2)$ action descends naturally to $\Lambda M$. Equation (\ref{energyfunctional}) defines a functional $E:\Lambda M\to\mathbb{R}$, whose critical points are smooth periodic or closed geodesics. However, under present considerations,  $E$ is $O(2)$-invariant which induces critical point degeneracy. Closed geodesics may be classified crudely by bifurcating isotropy groups. In particular, a constant geodesic corresponds to a point in $M$ with isotropy group $O(2)$, whereas the isotropy group of a non-constant geodesic is $\mathbb{Z}/k\mathbb{Z}$ with $1/k$ its minimal period for $k\in\mathbb{Z}^+$. 

The tangent space at a point $\gamma\in\Lambda M$, identified with a path, is $T_{\gamma}\Lambda M=H^1(\gamma^*TM)$, which is the space of sections of $\gamma^*TM$ of Sobolev class $H^1$ for $\gamma:S^1\to M$ and $\gamma^*:T^*M\to T^*S^1$ its pullback. The free loop space $\Lambda M$ of Sobolev class $H^1$ is a Hilbert manifold with respect to a $H^1$-inner product 
\begin{equation}
\langle\xi,\eta\rangle_1:=\int_{S^1}\langle\xi,\eta\rangle+\int_{S^1}\langle\nabla_t\xi,\nabla_t\eta\rangle=\langle\xi,\eta\rangle_0+\langle\nabla_t\xi,\nabla_t\eta\rangle_0.
\end{equation} 
The Arzel\'a-Ascoli theorem implies a compact inclusion $\Lambda M=H^1(S^1,M)\hookrightarrow C^0(S^1,M)$. This inclusion may be used to prove the following result.
\begin{proposition}[Klingenberg \cite{klingenberg}]
The Riemannian metric on $\Lambda M$ given by the $H^1$-inner product is complete.
\end{proposition}

\section{Morse Theory\label{sec:morsetheory}}

The gradient $\nabla E$ is a vector field on $\Lambda M$ defined by the following inner product.
\begin{equation}
\langle\nabla E[\gamma],\xi\rangle_1=dE[\gamma]\xi=\langle\dot\gamma,\nabla_t\xi\rangle_0, \quad \xi\in H^1(\gamma^*TM)
\end{equation} 
Furthermore, if $\gamma$ is smooth then $\langle\dot\gamma,\nabla_t\xi\rangle_0=-\langle\nabla_t\dot\gamma,\xi\rangle_0$ such that $\nabla E[\gamma]\in\Gamma(\gamma^*TM)$ is the unique periodic solution of the differential equation: 
\begin{equation}
\nabla_t^2\eta(t)-\eta(t)=\nabla_t\dot\gamma(t).
\end{equation} 
The following condition of Palais and Smale is necessary to extend Morse theory to the infinite-dimensional setting of Hilbert manifolds $\Lambda M$.
\begin{theorem}[Palais-Smale \cite{palais-smale}]\label{Palais-Smale} The energy functional $E:\Lambda M\to\mathbb{R}$ satisfies condition (C) of Palais and Smale:

(C) Let $(\gamma_m)\in\Lambda M$ be a sequence such that $E[\gamma_m]$ is bounded and $\|\nabla E[\gamma_m]\|_1\to 0$. Then $(\gamma_m)$ has limit points and every limit point is a critical point of $E$.
\end{theorem}

Let $\text{Crit}($E$):=\{\gamma\in\Lambda M:dE[\gamma]=0\}$ denote the set of critical points of the energy functional, i.e., the set of closed geodesics on $M$. For $a\ge 0$, we denote
\begin{align*}
   \Lambda^{\le a}:&=\{\gamma\in\Lambda M:E[\gamma]\le a\}, \\
   \Lambda^{<a}:&=\{\gamma\in\Lambda M:E[\gamma]<a\},\\
\intertext{and}
\Lambda^a:&=\{\gamma\in\Lambda M:E[\gamma]=a\}.
\end{align*} 
 The following properties of the negative gradient flow $\frac{d}{ds}\phi_s=-\nabla E(\phi_s)$ of the energy functional are due to condition (C) of Palais Smale. 

\begin{enumerate}[label=(\roman*)]
\item Crit($E$)$\cap\Lambda^{\le a}$ is compact for all $a\ge 0$.
\item the flow $\phi_s$ is defined for all $s\ge 0$.
\item given an interval $[a,b]$ of regular values, there exists a real number $s_0\ge 0$ such that $\phi_s(\Lambda^{\le b})\subset\Lambda^{\le a}$ for all $s\ge s_0$.
\item $\Lambda^0\equiv M\subset\Lambda^{\le\varepsilon}$ is a strong deformation retract via $\phi_s$ for $\varepsilon>0$ sufficiently small.
\end{enumerate}

Suppose $\gamma\in \text{Crit(}E\text{)}$ is a critical point. The $O(2)$-invariance of the energy functional $E$ means that the orbit $O(2)\cdot \gamma $ is contained entirely in Crit($E$).
\begin{definition} The $\textit{index}$ $\lambda(\gamma)$ of $\gamma$ is the dimension of the negative eigenspace of the Hessian $d^2E[\gamma]$.
\end{definition}
\begin{definition}
The nullity $\nu(\gamma)$ of $\gamma$ is the dimension of the kernel of $d^2E[\gamma]$, i.e. $\nu(\gamma):=\dim\ker d^2E[\gamma]$.
\end{definition} 
Note, $\langle\dot\gamma\rangle\in \ker d^2E[\gamma]$ so $\nu(\gamma)\ge 1$ if $\gamma$ is non-constant because Crit($E$) is invariant under the $S^1$ reparameterization action. The Hessian of $E$ at $\gamma$ is given by the second variation formula:
\begin{equation}
d^2E[\gamma](\xi,\eta)=-\int\langle\xi,\nabla_t^2\eta+R(\dot\gamma,\eta)\dot\gamma\rangle.
\end{equation}
\begin{definition}\label{morsefunction}
A $C^2$-function $f$ defined on a Hilbert manifold is a Morse function if all of its critical points are non-degenerate, which means that the Hessian has a $0$-dimensional kernel at each critical point $c$ so $\dim\ker d^2f(c)=0$.
\end{definition}
The energy functional $E$ defined on $\Lambda M$ will never satisfy the conditions of Definition \ref{morsefunction} because the critical points at level $0$ form a closed $\dim M$-dimensional manifold, which means that it is never non-degenerate. In a similar vein, the energy functional $E$ is $S^1$-invariant which means that the kernel of the Hessian of $E$ at a non-constant geodesic $\gamma$ is always at least $1$-dimensional, $\dim\ker d^2E[\gamma] \ge 1$, because it contains the infinitesimal generator $\dot\gamma$ of the $S^1$-action.
\begin{definition}[Oancea \cite{alexandru}] 
A $C^2$-function defined on a Hilbert manifold is said to be a Morse-Bott function if its critical set is a disjoint union of closed connected submanifolds $\bigsqcup_i\Sigma_i$ and, for each critical point $c_i$, the kernel of the Hessian at that critical point coincides with the tangent space to the respective connected component of the critical locus, $\ker d^2f(c_i)\cong T\Sigma_i$. As such, the critical set is said to be non-degenerate. 
\end{definition}
\begin{definition}
The index $\lambda(p)$ of a critical point $p$ is defined as $\iota(p)=\dim\{V_{\text{max}}\subset T_pM: d^2f(p)<0\}$, the dimension of the maximal subspace of the tangent space at $p$ on which the Hessian is negative definite.
\end{definition}

\begin{definition} 
The nullity $\nu(p)$ of a critical point $p$ is defined as $\nu(p)=\dim\ker d^2f(p)$, the dimension of the kernel of the Hessian at $p$.
\end{definition} 
\begin{theorem}[Palais \cite{palais-hilbertmanifolds}]\label{criticalmorse}
Assume that the Riemannian metric on $M$ is chosen in such a way that the energy functional $E$ on $\Lambda M$ is Morse-Bott.
\begin{enumerate}[label=(\roman*)]
\item The critical values of $E$ are isolated and there are only a finite number of connected components of Crit($E$) on each critical set.
\item The index and nullity of each connected component of Crit($E$) are finite.
\item If there are no critical values of $E$ in $[a,b]$ then $\Lambda^{\le b}$ retracts onto and is diffeomorphic to $\Lambda^{\le a}$.
\item Let $a<c<b$ and assume $c$ is the only critical value of $E$ in the interval [a,b]. Denote by $\Sigma_1,\cdots,\Sigma_r$ the connected components of Crit($E$) at level $c$, and denote by $\lambda_1,\cdots,\lambda_r$ their respective indices. \begin{itemize}
  \item Each manifold $\Sigma_i$ carries a well-defined vector bundle $\nu^{-}\Sigma_i$ of rank $\lambda_i$ consisting of negative directions for $d^2E|_{\Sigma_i}$.
  \item The sublevel set $\Lambda^{\le b}$ retracts onto a space homeomorphic to $\Lambda^{\le a}$ with the disk bundles $D\nu^{-}\Sigma_i$ disjointly attached to $\Lambda^{\le a}$ along their boundaries.
\end{itemize}
\end{enumerate}
\end{theorem}
We denote the disk bundles by $D\nu^-\Sigma_i$, which are associated to a fixed scalar product on the fibers $\nu^-\Sigma_i$, a sub-bundle of the normal bundle $T\Sigma_i^{\perp}$ with inner product given by the Hilbert structure of the manifold $\Lambda M$. According to \cite{alexandru}, there exist smooth embeddings $\phi_i:\partial D\nu^-\Sigma_i\to\partial\Lambda^{\le a}$ with disjoint images from which we can construct a quotient space $\Lambda^{\le a}\cup\bigcup_iD\nu^{-}\Sigma_i/\sim$ such that a point lying in the boundary $\partial D\nu^{-}\Sigma_i$ can be identified with the image under $\phi_i$ in $\partial\Lambda^{\le a}$.
\begin{remark}
When one passes the critical level $c$, the sublevel $c+\varepsilon$, for $\varepsilon>0$ sufficiently small, deformation retracts onto the the union of the sublevel $c-\varepsilon$ and a $\lambda(c)$-cell. The $\lambda$-cell is the negative vector bundle over the critical manifold consisting of a single point. Such a retraction is forced by a modification of the negative gradient flow $\frac{d}{ds}\phi_s=-\nabla f(\phi_s)$ of $f$.
\end{remark} 
\begin{corollary}
Adopting the the aforementioned a posteriori hypotheses, the sublevel set $\Lambda^{\le b}$ is diffeomorphic to and deformation retracts onto $\Lambda^{\le c}$. Similarly, the sublevel set $\Lambda^{<c}$ is diffeomorphic to and deformation retracts onto $\Lambda^{\le a}$. In particular, let $c\in[a,b]$ be a non-degenerate critical point of $E$ of index $\lambda(c)$ such that $E[c]=r$. Furthermore, suppose $E^{-1}[r-\varepsilon,r+\varepsilon]$ is compact and contains no critical points other than $c$. Then it follows that $\Lambda^{\le r+\varepsilon}$ is homotopy equivalent to the union of $\Lambda^{\le r-\varepsilon}$ and a $\lambda(c)$-cell.
\end{corollary}
\begin{corollary}
Suppose $E:\Lambda M\to\mathbb{R}$ is a real-valued smooth functional on $\Lambda M$, $a<b$, $E^{-1}[a,b]$ is compact, and there are no critical values between $a$ and $b$. Then the sublevel set $\Lambda^{\le b}$ is diffeomorphic to and deformation retracts onto $\Lambda^{\le a}$.
\end{corollary}
The above provides a useful heuristic for determining the homology groups of $\Lambda M$ by way of the non-decreasing filtration $\Lambda^{\le a}$, whereby \begin{equation} \Lambda M=\varinjlim_{a\to\infty}\Lambda^{\le a}. \end{equation} Thus, the homology of $\Lambda M$ is $H_{\bullet}(\Lambda M)=\varinjlim H_{\bullet}(\Lambda^{\le a})$. Suppose $E$ is non-degenerate such that the sequence $0=c_0<c_1<c_2<\cdots$ of its critical values is diverging. Then the following holds \cite{alexandru}.
\begin{enumerate}[label=(\roman*)]
\item If $c_i<a<c_{i+1}$, an inclusion $\Lambda^{\le c_i}\hookrightarrow\Lambda^{\le a}$ induces an isomorphism on homology $H_{\bullet}(\Lambda^{\le c_i})\cong H_{\bullet}(\Lambda^{\le a})$ which means the homology groups of $\Lambda^{\le a}$ change when passing a critical value.
\item The change in the homology upon crossing a critical level $c$ of $E$ is entirely captured by the long exact sequence of the pair $(\Lambda^{\le c},\Lambda^{<c})$:
\begin{equation*}
\cdots\longrightarrow H_{\bullet}(\Lambda^c)\longrightarrow H_{\bullet}(\Lambda^{\le c})\longrightarrow H_{\bullet}(\Lambda^{\le c},\Lambda^{<c})\stackrel{[-1]}{\longrightarrow}\cdots
\end{equation*}
\item If there are no critical values of $E$ in $[a,b]$ then $\Lambda^{\le b}$ retracts onto and is diffeomorphic to $\Lambda^{\le a}$. 
Let $\Sigma^c:=\bigsqcup_{i=1}^r \Sigma_i$ be the critical set of $E$ at level $c$. Therefore, the following isomorphisms hold. 
\begin{equation*}
H_{\bullet}(\Lambda^{\le c},\Lambda^{<c})\cong H_{\bullet}\left(\bigsqcup_iD\nu^-\Sigma_i,\bigsqcup_i\partial D\nu^-\Sigma_i\right)\cong\bigoplus_iH_{\bullet-\lambda_i}(\Sigma_i;o_{\nu^-\Sigma_i})
\end{equation*} Note, $H_{\bullet-\lambda_i}(\Sigma_i;o_{\nu^-\Sigma_i})$ denotes the homology groups of $\Sigma_i$ with coefficients given by the local system of orientations of the bundle $\nu^-\Sigma_i$. The first isomorphism follows from excision and (iv) of Theorem \ref{criticalmorse}. Likewise, the second isomorphism is induced by the Thom isomorphism for finite rank vector bundles. 
\end{enumerate}
We present the useful and well-known theorem of Gromoll and Meyer.
\begin{theorem}[Gromoll-Meyer \cite{gromoll-meyer}]\label{Gromoll/Meyer} Let $M$ be a simply connected closed manifold. If the sequence $\{b_k(\Lambda M;\mathbb{F}_p)\}_{k\ge 0}$ is unbounded for some prime $p\ge 2$, then any Riemannian metric on $M$ admits infinitely many prime closed geodescis.
\end{theorem}

It was shown by Vigu\'e-Poirrier and Sullivan \cite{poirrier-sullivan} that for a simply connected closed manifold $M$, the sequence of rational Betti numbers $b_k(\Lambda M;\mathbb{Q})=\text{rank }H^k(\Lambda M;\mathbb{Q})$ is unbounded if and only if the rational cohomology ring $H^*(M;\mathbb{Q})=\bigoplus_{k\in\mathbb{N}}H^k(M;\mathbb{Q})$ has at least two generators. That is, $M$ does not have rational cohomology that is the quotient of a polynomial ring. There are a few simply connected compact globally symmetric spaces of rank $>1$ whose rational cohomology ring $H^*(M;\mathbb{Q})$ is not a truncated polynomial ring. In fact, it was shown by Ziller \cite{ziller} that compact globally symmetric spaces of rank $>1$ are such that the sequence of Betti numbers for their free loop space is unbounded with coefficients in $\mathbb{F}_2$. On the other hand, compact globally symmetric spaces of rank $1$ have rational cohomology generated by a single element.

We must address two complications that arise in view of Gromoll and Meyer \ref{Gromoll/Meyer}, the first being that iterates $\gamma_m:S^1\to M$ with $\gamma_m(t)=\gamma(mt), m\in\mathbb{Z}$ of a closed geodesic $\gamma$ are themselves closed geodesics (under the action of infinite-dimensional reparameterization), and thereby constitute a  part of the homology of free loop space $\Lambda M$. Second, we would like to handle degeneracy for closed geodesics. It was shown by Gromoll and Meyer \cite{gromoll-meyer} that a single degenerate closed geodesic adds to homology by at most $\nu(\gamma)$ for degrees contained in $[\lambda(\gamma),\lambda(\gamma)+\nu(\gamma)]$. This result follows from the lemmata which dictate the behavior of index and nullity subject to iteration.
\begin{lemma}[Gromoll-Meyer Index Iteration \cite{gromoll-meyer}]\label{Index Iteration}
The index is either $\lambda(\gamma_m)=0$ for all $m\ge 1$ or there is an $\varepsilon>0$ and $C>0$ such that for all $m,s\ge 1$:
\begin{equation}
\lambda(\gamma_{m+s})-\lambda(\gamma_m)\ge\varepsilon s-C.
\end{equation}
\end{lemma}
\begin{lemma}[Gromoll-Meyer Nullity Iteration \cite{gromoll-meyer}]\label{Nullity iteration} The nullity is either $\nu(\gamma_m)=1$ for all $m\ge 1$ or there are finitely many iterations $\gamma_{m_1},\cdots,\gamma_{m_s}$ of $\gamma$ such that for all $m\ge 1$ there is an $i$ and $j$ for which $m=jm_i$ and $\nu(\gamma_m)=\nu(\gamma_{m_i})$. That is, the sequence $\{\nu(\gamma_m)\}_{m\ge 1}$ only takes on finitely many values. 
\end{lemma}

By virtue of a combinatorial argument and the Gromoll-Meyer index and nullity iteration lemmata, if the number of simple closed geodesics is bounded, there is a bound on the Betti numbers of $\Lambda M$.

Let
\begin{equation}
A_{\gamma}:H^1(\gamma^*TM)\to L^2(\gamma^*TM),\quad \xi\mapsto -\xi''-R(\dot\gamma,\xi)\dot\gamma
\end{equation} be the self-adjoint $\textit{asymptotic elliptic operator}$ at the closed geodesic $\gamma$ with compact resolvent such that $d^2E[\gamma](\xi,\eta)=\langle A_{\gamma}\xi,\eta\rangle_{L^2}$. More precisely, 
\begin{equation}
\begin{split}
 	\langle A_{\gamma}\xi,\eta\rangle_{L^2}&=\langle A_{\gamma}\eta,\xi\rangle_{L^2}\\ 
		&=\int\langle-\eta''-R(\dot\gamma,\eta)\dot\gamma,\eta\rangle \\
           &= -\int\langle\xi,\nabla_t^2\eta+R(\dot\gamma,\eta)\dot\gamma\rangle.
\end{split}
\end{equation} Its spectrum is discrete with no accumulation point besides $\pm\infty$ \cite{kato}. The eigenvectors of $A_{\gamma}$ are smooth by elliptic regularity. The eigenvalue problem $-\xi''-R(\dot\gamma,\xi)\dot\gamma=-\lambda\xi$ has no non-trivial periodic solutions for $\lambda\gg 0$. Moreover, the spectrum of the elliptic operator $A_{\gamma}$ is bounded from below which implies the finiteness of the index of $\gamma$, given by 
\begin{equation}
\lambda(\gamma)=\sum_{\text{Spec}(A_{\gamma})\ni\lambda<0}\text{mult}(\lambda)<+\infty
\end{equation} for $\text{mult}(\lambda)$ the multiplicity of the eigenvalue $\lambda$. The nullity of $\gamma$ is 
\begin{equation}
\nu(\gamma)=\dim\ker A_{\gamma},
\end{equation} for which we have the bounds $1\le\nu(\gamma)\le 2n-1$. The first inequality is due to $\langle\dot\gamma\rangle\in\ker A_{\gamma}$, and the second inquality follows because an element in $\ker A_{\gamma}$ solves a second-order ordinary differential equation.

It was shown by Takens and Klingenberg that for a $C^4$-generic metric on a closed manifold $M$, there exist closed geodesics of $\textit{twist type}$ for which the eigenvalues of the Poincar\'e return map have modulus unity, or there exist closed geodesics of $\textit{hyperbolic type}$ for which the eigenvalues of the Poincar\'e return map lie outside or inside the unit circle. The Birkhoff-Lewis fixed point theorem implies that for the former case there are infinitely many prime closed geodesics in a neighborhood of the twist geodesic. It is then left to show that this is true for hyperbolic geodesics. It should be noted that if all geodesic orbits are hyperbolic, then $\lambda(\gamma_m)=m\lambda(\gamma)$.
\begin{theorem}[Rademacher \cite{rademacher}]\label{Rademacher} 
Let $M$ be a simply connected closed Riemannian manifold with rational cohomology a truncated polynomial ring in one even variable and all geodesics hyperbolic. Then there are infinitely many geometrically distinct hyperbolic geodesics.
\end{theorem}
\begin{theorem} 
Let $M$ be a simply connected closed Riemannian manifold with rational homotopy type that of a rank $1$ symmetric space (i.e. $S^n,\mathbb{CP}^n$, $\mathbb{HP}^n$, $\mathbb{P}^2(\mathbb{O})$) and suppose all closed geodesics on $M$ are hyperbolic. Then the number $\pi(x)$ of prime closed geodesics of length less than or equal to $x$ is on the order of 
\begin{equation}
\pi(x)=|\{\gamma\in \Lambda M: L[\gamma]\le x\}| \sim\frac{x}{logx}.
\end{equation}
\end{theorem} Recall that the isotropy group at a non-constant closed geodesic is $\mathbb{Z}_m$ with $1/m$ the minimal period. Since $\mathbb{Z}_m,m\ge 2$ are $\mathbb{Q}$-acyclic, we have the following isomorphism on equivariant cohomology:
\begin{equation*} H_{\mathbb{Z}_m}^*(\Lambda M,M;\mathbb{Q})\cong H^*(\Lambda M/{\mathbb{Z}_m}, M;\mathbb{Q})
\end{equation*} where $H_{\mathbb{Z}_m}^*(\bullet):=H^*({\bullet}_{\mathbb{Z}_m})$ is $\mathbb{Z}_m$-equivariant cohomology.

The most significant consequence of $O(2)$-invariance for $E$ is that the $m$-fold iterate $\gamma_m$ of a closed geodesic adds to $H_{\bullet}^{O(2)}(\Lambda M)$ by $H_{\bullet}(B\mathbb{Z}_m)$. Thus \cite{brown}, 
\begin{align*}
       H_{\bullet}(B\mathbb{Z}_m;\mathbb{Z}) &=  \begin{cases} 
      \mathbb{Z}, & \text{if }\bullet=0, \\
      \mathbb{Z}_m, & \text{if} \hspace{.05 cm}\bullet\text{odd}, \\
      0, & \text{if }\bullet>0\text{ even}. 
   \end{cases}
\end{align*} By the universal coefficient theorem, we deduce that 
\begin{align*}
H_{\bullet}(B\mathbb{Z}_m;\mathbb{Z}_{\ell})&= \begin{cases} 
      \mathbb{Z}_{\ell}, & \text{if }\bullet=0, \\
      \mathbb{Z}_d, & \text{if }\bullet>0,
   \end{cases}
\end{align*} where $d=\text{gcd}(m,\ell)$, and 
\begin{align*}
H_{\bullet}(B\mathbb{Z}_m;\mathbb{Q})=  
\begin{cases} 
      \mathbb{Q}, & \text{if }\bullet=0, \\
      0, & \text{if }\bullet>0. 
   \end{cases}
\end{align*}

\section{Rational Homotopy Theory\label{sec: homotopytheory}}

An important consequence of rational homotopy theory is that every simply connected closed Riemannian manifold $M$ with rational cohomology not isomorphic to a truncated polynomial has infinitely many geometrically distinct prime closed geodesics.
\begin{definition}
Suppose $M,N$ are simply connected topological spaces. A continuous map $f:M\to N$ is called a rational homotopy equivalence if it induces an isomorphism on homotopy groups tensored with $\mathbb{Q}$.
\end{definition}
\begin{definition}
A rational space is a simply connected CW complex whose homotopy groups are vector spaces over $\mathbb{Q}$.
\end{definition}
\begin{definition}
The rationalization of a simply connected CW complex $M$ is the rational space $M_{\mathbb{Q}}$ for which the map from $M$ to $M_{\mathbb{Q}}$ descends to an isomorphism on rational homology, i.e. $\pi_k(M_{\mathbb{Q}})\cong\pi_k(M)\otimes\mathbb{Q}$ and, by the universal coefficient theorem, $H_k(M_{\mathbb{Q}},\mathbb{Z})\cong H_k(M,\mathbb{Z})\otimes\mathbb{Q}\cong H_k(M,\mathbb{Q})$ for all $k>0$.
\end{definition}
The rational cohomology ring $H^*(M,\mathbb{Q})$ is a graded commutative algebra and $\pi_*(M)\otimes\mathbb{Q}$ is a graded homotopy Lie algebra. Let $\Omega M$ be the $\textit{loop space}$ of $M$, then $\pi_*(\Omega M)\otimes\mathbb{Q}$ is also a graded Lie agebra by the Whitehead product structure. 
\begin{theorem}
The homology of loop space $H_{\bullet}(\Omega M,\mathbb{Q})$ is the (graded) universal enveloping algebra of the (graded) Lie algebra $\pi_{\bullet}(\Omega M,\mathbb{Q})\cong\pi_{\bullet-1}(M,\mathbb{Q})$.
\end{theorem}
\begin{definition}
Let $M$ be a simply connected (CW complex) space for $H^*(M,\mathbb{Q})$ finite-dimensional over $\mathbb{Q}$. A simply connected space $M$ is said to be rationally elliptic if the homotopy Lie algebra $\pi_*(M)\otimes\mathbb{Q}$ is finite-dimensional over $\mathbb{Q}$, or else it is rationally hyperbolic.
\end{definition}

According to F\'elix and Halperin \cite{felix-halperin}, if $M$ is rationally hyperbolic, there exists an $\alpha>1$ and an integer $N$ such that for all $n\ge N$,
\begin{equation}
\sum_{i=1}^n\dim_{\mathbb{Q}}\pi_i(M)\otimes\mathbb{Q}\ge \alpha^n.
\end{equation}
\begin{remark}
Let $M$ be a simply connected closed Riemannian manifold. Then the rational cohomology ring $H^*(M;\mathbb{Q})$ requires more than one generator if and only if $\pi_{\text{odd}}(M)\otimes\mathbb{Q}$ is more than $1$-dimensional, i.e. $\dim\pi_{\text{odd}}(M)\otimes\mathbb{Q}>1$.
\end{remark}
\begin{proof}
A basis for rational homotopy $\pi_k(M)\otimes\mathbb{Q}$ corresponds to a set of algebra generators of the minimal model. Then the claim follows from Proposition 1 of Vigu\'e-Poirrier and Sullivan \cite{poirrier-sullivan}.
\end{proof}
\begin{corollary}
The rational space $M_{\mathbb{Q}}$ cannot be a $K(G,n)$ Eilenberg--MacLane space for $G$ singly generated if $\dim\pi_{\text{odd}}(M)\otimes\mathbb{Q}>1$.
\end{corollary}
\begin{proof}
Suppose the rational space $M_{\mathbb{Q}}$ is a $K(G,n)$ Eilenberg--MacLane space. Then for $n$ a positive integer, $\pi_n(K(G,n))\cong G$, and $\pi_i(M_{\mathbb{Q}})=0$ for all $i\ne n$. Thus, $\dim\pi_i(M_{\mathbb{Q}})=\dim\pi_i(M)\otimes\mathbb{Q}=\delta_{in}$, which implies that $\dim\pi_{\text{odd}}(M)\otimes\mathbb{Q}\le 1$. It follows that $H^*(M;\mathbb{Q})$ is singly generated so $M_{\mathbb{Q}}\ne K(G,n)$, a contradiction. 
\end{proof} 
\begin{lemma}
If the rational cohomology ring $H^*(M;\mathbb{Q})$ is not singly generated, $M_{\mathbb{Q}}=M\times E\mathbb{Q}\ne K(G,n)$ for any positive integer $n$ or group $G$ acting freely on $M$.
\end{lemma}
\begin{theorem}[Hurewicz Theorem]\label{hurewicz}
If $X$ is an $(k-1)$-connected topological space, then for a positive integer $n$, the group homomorphism
\begin{equation*}
f_*:\pi_n(X)\to H_n(X)
\end{equation*} is an isomorphism for all $n\le k$ and an abelianization for $n=1$.
\end{theorem}
\begin{definition}
Let $G$ be an abelian group. The group $G \otimes_{\mathbb{Z}} \mathbb{Q}$ is the divisible hull of $G$. The dimension of  $G \otimes_{\mathbb{Z}} \mathbb{Q}$ over $\mathbb{Q}$ is the rational rank of $G$. We denote by $\text{rank }_{\mathbb{Q}}(G)$ the rational rank of $G$. If $G$ is torsion free then the map $G\to G\otimes\mathbb{Q}$ is injective and the rank of $G$ is the minimum dimension of the $\mathbb{Q}$-vector space containing $G$ as an abelian subgroup.
\end{definition} 
\begin{lemma}\label{rathyperbolic}
If $M$ is a rationally hyperbolic manifold then $\dim_{\mathbb{Q}}\pi_1(M)\otimes\mathbb{Q}\ge 2$ and $\dim_{\mathbb{Q}}\pi_k(M)\otimes\mathbb{Q}>\alpha^{k-1}$ for an integer $k>1$ and a real number $\alpha>1$. Moreover, $\dim_{\mathbb{Q}}\pi_{\text{odd}}(M)\otimes\mathbb{Q}>1$ so the rational cohomology ring $H^*(M;\mathbb{Q})$ has at least two generators.
\end{lemma}
\begin{proof}
If $M$ is rationally hyperbolic, there exists a real number $\alpha>1$ and an integer $N$ such that 
\begin{equation*}
\sum_{i=1}^n\dim_{\mathbb{Q}}\pi_i(M)\otimes\mathbb{Q}\ge \alpha^n
\end{equation*} for all $n\ge N$. If $n=1$ then $\dim_{\mathbb{Q}}\pi_1(M)\otimes\mathbb{Q}\ge \alpha$ for $\alpha>1$ so $\dim_{\mathbb{Q}}\pi_1(M)\otimes\mathbb{Q}>1$. If $n=2$ then $\dim_{\mathbb{Q}}\pi_1(M)\otimes\mathbb{Q}+\dim_{\mathbb{Q}}\pi_2(M)\otimes\mathbb{Q}\ge \alpha^2$. Assume $\dim_{\mathbb{Q}}\pi_1(M)\otimes\mathbb{Q}= \alpha$ so that $\dim_{\mathbb{Q}}\pi_2(M)\otimes\mathbb{Q}\ge \alpha^2-\alpha$. If $n=3$ then $\dim_{\mathbb{Q}}\pi_1(M)\otimes\mathbb{Q}+\dim_{\mathbb{Q}}\pi_2(M)\otimes\mathbb{Q}+\dim_{\mathbb{Q}}\pi_3(M)\otimes\mathbb{Q}\ge \alpha^3$. Let $\dim_{\mathbb{Q}}\pi_2(M)\otimes\mathbb{Q}=\alpha^2-\alpha$ so $\alpha+(\alpha^2-\alpha)+\dim_{\mathbb{Q}}\pi_3(M)\otimes\mathbb{Q}\ge \alpha^3$ or $\dim_{\mathbb{Q}}\pi_3(M)\otimes\mathbb{Q}\ge \alpha^3-\alpha^2$. It follows by induction that $\dim_{\mathbb{Q}}\pi_k(M)\otimes\mathbb{Q}\ge \alpha^k-\alpha^{k-1}$ for $\alpha>1$. More generally, $\dim_{\mathbb{Q}}\pi_k(M)\otimes\mathbb{Q}\ge \alpha^{k-1}(\alpha-1)>\alpha^{k-1}$ or $\dim_{\mathbb{Q}}\pi_k(M)\otimes\mathbb{Q}> \alpha^{k-1}$ for $\alpha>1$. As per the above, for $k=1$, it is already known that $\dim_{\mathbb{Q}}\pi_1(M)\otimes\mathbb{Q}>1$. 
\end{proof}

Another program for determining the dimension of $\pi_i(M)\otimes\mathbb{Q}$ is by way of minimal models. If $M$ is a closed, simply connected smooth manifold, then the rank of the rational homotopy groups $\pi_i(M)\otimes\mathbb{Q}$ equals the number of degree $i$ generators introduced in the construction of the minimal model for $M$. For manifolds that are $\textit{formal}$, e.g. if $M$ is K\"{a}hler, their cohomology ring with trivial differential is quasi-isomorphic as a differential graded algebra to the minimal model. Thus, we can use the cohomology ring to determine the rank of rational homotopy groups.

\begin{theorem}\label{betti-inequality}
Let $E:\Lambda M\to\mathbb{R}$ be a Morse function on the free loop space of Sobolev class $H^1$. Let $C_{\lambda}$ denote the number of critical points of the energy functional $E$ of index $\lambda$. Then we have the following constraints:
\begin{equation*}
\begin{split}
& \text{rank } H_{\lambda}(\Lambda M)=b_{\lambda}\le C_{\lambda} \text{ for each } \lambda,\\ &
\chi(\Lambda M)=\sum_{\lambda}(-1)^{\lambda}C_{\lambda}.
\end{split}
\end{equation*}
\end{theorem}
\begin{proof}
Let $\Lambda_1$ be obtained from $\Lambda_0$ by attaching a handle of index $\lambda$ for $M_0\subset M_1$, then $H_k(\Lambda_1,\Lambda_0)=H_k(D^{\lambda},S^{\lambda-1})$. Thus, $H_k(\Lambda_1,\Lambda_0)=\mathbb{Z}$ if $k=\lambda$ and $0$ otherwise. Denote by $c_1<c_2<\dots<c_n$ the critical values of $E$. Consider the sequence of real numbers $\{a_i\}$ such that $c_i<a_i<c_{i+1}$ for $0\le i\le n-1$, and let $\Lambda_i=\Lambda^{\le a_i}$. Thus, we obtain the inclusions of an increasing sequence of manifolds with boundary, $D^{\dim \Lambda M}=\Lambda_0\subset \Lambda_1\subset\Lambda_2\subset\dots\subset\Lambda_n=\Lambda M$, where $\Lambda_{i+1}$ is $\Lambda_i$ with finitely many cells attached for each critical point in the set $\Lambda^{[a_i,a_{i+1}]}$. Since homology groups $H_k$ are sub-additive, it follows that 
\begin{equation*}
\text{rank } H_k(\Lambda M;\mathbb{Z})\le\sum_i\text{rank } H_k(\Lambda_{i+1},\Lambda_i)
\end{equation*} so the rank of $H_k(\Lambda M;\mathbb{Z})$ is less than or equal to the number of critical points of index $k$.
\end{proof}

Suppose $\Lambda M$ is elliptic with top non-zero rational homology group in degree $n$. Then the Euler characteristic of the infinite-dimensional CW complex $\Lambda M$ is $\chi(\Lambda M)=\sum_{\lambda=0}^{\infty}(-1)^{\lambda}C_{\lambda}$ and $b_{\lambda}\le C_{\lambda}$. Furthermore, $b_{\lambda}\le{n\choose \lambda}$ so $C_{\lambda}={n\choose \lambda}$ as demonstrated in the following lemma.
\begin{lemma}
If $\Lambda M$ is an elliptic space with top non-zero rational cohomology in degree $n$, then the rational cohomology ring $H^*(\Lambda M;\mathbb{Q})$ is the complete intersection ring and does not have at least two generators because the sequence of Betti numbers $\{b_{\lambda}(\Lambda M;\mathbb{Q})\}_{\lambda\ge 0}$ is bounded by  $\{{n\choose\lambda}\}_{\lambda\ge 0}$.
\end{lemma}
Let $M$ be rationally elliptic so that $b_{2k+1}(M)=0$, for $k\in\mathbb{N}$, if $\chi(M)>0$ and let $\pi_i(M)\otimes\mathbb{Q}=0$ for $i\le\dim M-1$. Let $n:=\dim M$ such that 
\begin{equation}
\begin{split}
\chi(M)&=\sum_{\lambda=0}^{\infty}(-1)^{\lambda}\text{rank }H_{\lambda}(M;\mathbb{Q})\\
&=\sum_{\lambda=0}^{\infty}(-1)^{\lambda}\dim\pi_{\lambda}(M)\otimes\mathbb{Q}\\&
=(-1)^n\dim\pi_n(M)\otimes\mathbb{Q}>0.
\end{split}
\end{equation} Recall the isomorphism $\pi_n(\Lambda M;\mathbb{Q})\cong\pi_{n+1}(M;\mathbb{Q})$, which implies $\pi_n(\Lambda M_{\mathbb{Q}})=\pi_n(\Lambda M)\otimes\mathbb{Q}\cong\pi_{n+1}(M)\otimes\mathbb{Q}\cong H_{n+1}(M)\otimes\mathbb{Q}$ where it is assumed that $\pi_i(M)\otimes\mathbb{Q}=0$ for $i\le n-1.$ Then 
\begin{equation}
\begin{split}
H_{n-1}(\Lambda M_{\mathbb{Q}};\mathbb{Z})\cong H_{n-1}(\Lambda M;\mathbb{Q})&\cong\pi_{n}(M;\mathbb{Q})\cong\pi_n(M)\otimes\mathbb{Q}\\&\cong H_n(M)\otimes\mathbb{Q}\cong H_n(M;\mathbb{Q}),
\end{split}
\end{equation} where the first isomorphism follows from the rational Hurewicz theorem and the last by $H_i(M_{\mathbb{Q}};\mathbb{Z})\cong H_i(M;\mathbb{Z})\otimes\mathbb{Q}\cong H_i(M;\mathbb{Q})$. We a priori assume the odd Betti numbers $b_{2i+1}$ vanish. Then we obtain the following bounds:
\begin{equation}
\begin{split}
0<\chi(M)&=\sum_{i=0}^{\dim M}(-1)^ib_i(M;\mathbb{Q})\\
&=\sum_{i=1}^{\dim M}(-1)^{i-1}b_{i-1}(\Lambda M;\mathbb{Q})\\
&\le\sum_{i=0}^{\infty}(-1)^i\text{rank }H_i(\Lambda M;\mathbb{Q})=\chi(\Lambda M)
\end{split}
\end{equation} so $0<\chi(M)\le\chi(\Lambda M)$ or $\chi(\Lambda M)>0$. 
\begin{theorem}
Let $M$ be a closed $(\dim M-1)$-connected Riemannian manifold. If the Euler characteristic $\chi(\Lambda M)$ of free loop space of Sobolev type $H^1=W^{1,2}$ is positive, then $M$ is rationally elliptic. 
\end{theorem}
Bott's conjecture posits that any simply connected closed Riemannian manifold with nonnegative sectional curvature should be rationally elliptic. Assuming Bott's conjecture holds, if the sectional curvature of $M$ is nonnegative, i.e. $K\ge 0$, then $M$ is rationally elliptic. As before, let $M$ be a closed orientable even $n$-dimensional Riemannian manifold without boundary, and suppose $\Omega$ is the curvature form of the Levi-Civita connection $^{\text{L.C.}}\nabla$ associated with $M$ of $\dim M:=n=2k$, $k\in\mathbb{N}$. Then $\Omega$ is an $\frak{so}(n)$-valued $2$-form on $M$, so it is a skew-symmetric $n$-by-$n$ matrix of $2$-forms over $\wedge^{\text{even}}T^*M$. Let $Pf(\Omega)$ denote the $n$-form Pfaffian. Since the sectional curvature is nonnegative, by the generalized Gauss-Bonnet theorem,
\begin{equation}
\begin{split}
\chi(M)&=\frac{1}{\sqrt{(2\pi)^n}}\int_M Pf(\Omega)\\&=\frac{1}{\sqrt{(2\pi)^n}}\int_Mi^{\frac{n^2}{4}}\exp\left(\frac{1}{2}\text{Tr}\log((\sigma_y\otimes I_{n/2})^T\cdot\Omega) \right)\ge 0.
\end{split}
\end{equation}

The prime geodesic theorem describes the asymptotic distribution of prime geodesics on an $n$-dimensional hyperbolic manifold $M=\mathbb{H}^n/\Gamma$ with $\Gamma$ a discrete subgroup of $SO_{(1,n)}^{+}\mathbb {R}$.
\begin{theorem} [Prime Geodesic Theorem \cite{sarnak}]
Let $M$ be a hyperbolic manifold of dimension $\dim M=m+1$. Furthermore, let $\Gamma=\pi_1(M)$ be its fundamental group. For any element $\gamma\in\Gamma$ there exists a closed geodesic representative $\gamma$ in $M$. Let $\ell(\gamma)$ denote the length of the geodesic $\gamma$, $N(\gamma):=e^{\ell(\gamma)}$ its norm, and $\pi_{\Gamma}(x)$ the number of primitive elements $\gamma\in\Gamma$ such that $N(\gamma)\le x$. Moreover, let $\{z_1,\dots,z_N\}$ denote the zeros of the Selberg zeta function and $\text{li}(x)=\int_0^x\frac{dt}{lnt}$ the logarithmic integral. Then $\pi_{\Gamma}(x)$ satisfies the following equality.
\begin{equation}
\pi_{\Gamma}(x)=\text{li}(x^m)+\sum_{n=0}^Nli(x^{z_n})+(\text{error term})
\end{equation}
\end{theorem}
As before, let $C_{\lambda}$ denote the number of critical points of $E:\Lambda M\to\mathbb{R}$ of index $\lambda$ where the index of a critical point $\gamma\in\Lambda M$ is $\lambda(\gamma):=\dim\{V_{\text{max}}\subset T_{\gamma}\Lambda M:d^2E<0\}<+\infty$. Consider the energy functional $E:\Lambda M\to\mathbb{R}$ for $\Lambda M$ an infinite-dimensional Hilbert manifold. For $a,b\in\mathbb{R}$ regular values of $E$, let $C_{\lambda}(a,b)$ denote the number of critical points of index $\lambda$ in $E^{-1}[a,b]$ with homology rank $r_{\lambda}(a,b)=\text{rank }(H_{\lambda}({\Lambda M}^b,{\Lambda M}^a))$ and torsion rank $t_{\lambda}(a,b)=t(H_{\lambda}({\Lambda M}^b,{\Lambda M}^a))$ for ${\Lambda M}^c=E^{-1}(-\infty,c]$. That is,
\begin{equation}
\begin{split}
& r_{\lambda}(a,b)+t_{\lambda}(a,b)+t_{\lambda -1}(a,b)\le C_{\lambda}, \\
& \sum_{i=0}^{\lambda}(-1)^{\lambda-i}r_i(a,b)\le\sum_{i=0}^{\lambda}(-1)^{\lambda-i}C_i,
\end{split}
\end{equation} for $\lambda=0,1,\dots$. For sufficiently large $\lambda$, the last inequality becomes an equality. Thus, for $\chi(\Lambda M)=\sum_{\lambda=0}^{\infty}(-1)^{\lambda}b_{\lambda}(
\Lambda M; \mathbb{Q})$ and $b_{\lambda}+t_{\lambda}+t_{\lambda-1}\le C_{\lambda}$, we have $b_{\lambda}\le C_{\lambda}$. Therefore the total number of primitive elements $\gamma\in\Gamma$ is
\begin{equation*}
\lim_{t\to\infty^-}\int_0^t\pi_{\Gamma}(x)dx=\int_0^{\infty}\pi_{\Gamma}(x)dx.
\end{equation*} Recall, the total number of critical points is given by the sum over all indices of the number of critical points of a specific index, i.e., $\sum_{\lambda}C_{\lambda}$. As such,
\begin{equation}
\sum_{\lambda}C_{\lambda}=\sum_{\lambda}b_{\lambda}(\Lambda M;\mathbb{Q})=\int_0^{\infty}\pi_{\Gamma}(x)dx.
\end{equation} That is, for $\Lambda M=\mathbb{H}^{\infty}/\Gamma$ hyperbolic, 
\begin{equation*}
\begin{split}
\sum_{\lambda=0}^{\dim\Lambda M}b_{\lambda}(\Lambda M;\mathbb{Q})&=\int_0^{\infty}\text{li}(x^{\dim \Lambda M-1})dx+\sum_{n=0}^N\text{li}(x^{z_n})dx+(\text{error term})dx \\
&=\lim_{k\to\infty}\int_0^{\infty}\text{li}(x^{k-1})dx+\sum_{n=0}^N\int_0^{\infty}\text{li}(x^{z_n})dx+\int_0^{\infty}(\text{error term})dx.
\end{split}
\end{equation*} Moreover, for $\Lambda M=\mathbb{H}^{\infty}/\Gamma$ hyperbolic with $\Gamma\subset SO^+_{(1,\infty)}\mathbb{R}$, the \textit{Poincar\'e polynomial} of $\Lambda M$ is the generating function of Betti numbers of $\Lambda M$, $P_{\Lambda M}(z)=b_0(\Lambda M)+b_1(\Lambda M)z+b_2(\Lambda M)z^2+\cdots$, so $P_{\Lambda M}(1)=\sum_{\lambda=0}^{\infty}b_{\lambda}(\Lambda M)=\int_0^{\infty}\pi_{\Gamma}(x)dx$.
\begin{lemma} By the main result \ref{main result}, 
\begin{equation*}
\int_0^{\infty}\pi_{\Gamma}(x)dx=P_{\Lambda M}(1)=+\infty.
\end{equation*}
Thus, there are infinitely many such prime closed geodesics.
\end{lemma}
We remark that we have only shown this to be true for $\Lambda M$ hyperbolic. Thus, we resort to a holonomic classification of simply connected manifolds due to Berger. 

\section{Holonomy, The Berger Classification, and The Classification of Symmetric Spaces \label{sec: holonomyclassification}}
We begin with preliminary definitions of reducibility and irreducibility of holonomy representations.
\begin{definition}
The Riemannian holonomy of a Riemannian manifold $(M,g)$ is the holonomy of the Levi-Civita connection on the tangent bundle of the manifold.
\end{definition}
\begin{definition}
A Riemannian manifold $(M,g)$ is said to be (resp. locally) reducible if it is (resp. locally) isometric to a Riemannian product.
\end{definition}
\begin{theorem}[de Rham \cite{deRham}]\label{de Rham} If a Riemannian manifold is complete, simply connected, and if its holonomy representation is reducible, then $(M,g)$ is a Riemannian product. 
\end{theorem}

Berger provided a complete classification of holonomy for simply connected Riemannian manifolds which are irreducible and non-symmetric (see Table \ref{table}).

\begin{table}[H]
\centering
    \begin{tabular}{ | c | c | c | c | }
    \hline
    $\text{Hol}(g)$ & $\dim_{\mathbb{R}}(M)$ & $G$-structure & Description \\ \hline
    $SO(n)$ & n & Orientable manifold & - \\ \hline
    $U(n)$ & $2n$ & K\"{a}hler & K\"{a}hler \\ \hline
    $SU(n)$ & $2n$ & Calabi-Yau Manifold & Ricci-flat, K\"{a}hler\\ \hline
    $Sp(n)\cdot Sp(1)$ & $4n$ & Quaternion-K\"{a}hler manifold & Einstein\\ \hline
    $G_2$ & $7$ & $G_2$ manifold & Ricci-flat\\ \hline
    $Spin(7)$ & $8$ & $Spin(7)$ manifold & Ricci-flat\\ \hline
    \end{tabular}
    \caption{The Berger classification of holonomy groups for irreducible and non-symmetric simply connected Riemannian manifolds.}
    \label{table}
\end{table}

There is a sequence of inclusions $Sp(n)\subset SU(2n)\subset U(2n)\subset SO(4n)$ so every hyperk\"{a}hler manifold is Calabi-Yau, every Calabi-Yau manifold is K\"{a}hler, and every K\"{a}hler manifold is also orientable. 
\begin{theorem}\label{holonomytheorem}
If $M$ is simply connected and has reducible holonomy, then the de Rham decomposition theorem implies that $M$ is a product, and hence does not have rational cohomology generated by one element.
\end{theorem}
\begin{proof}
Let $M$ be a simply connected $n$-dimensional manifold such that there is a complete decomposable reduction of the tangent bundle $TM=T^{(0)}M\oplus T^{(1)}M\oplus\cdots\oplus T^{(k)}M$ under the action of holonomy and the tangent bundles $T^{(i)}M$ parameterizing Frobenius-integrable distributions. If $\text{Hol}(M)$ is reducible then $M=V_0\times V_1\times\cdots\times V_k$ for each $V_i$ an integral manifold for the respective tangent bundle $T^{(i)}M$ where $V_0$ is an open subset of $\mathbb{R}^n$. Moreover, the holonomy group $\text{Hol}(M)$ splits as a direct product of the holonomy groups of each $M_i$. Thus, if $M=M_0\times M_1\times\cdots\times M_k$ then $\text{Hol}(M)=\text{Hol}(M_0\times M_1\times\cdots\times M_k)=\text{Hol}(M_0)\times \text{Hol}(M_1)\times\cdots\times \text{Hol}(M_k)$. So $M=M_0\times M_1\times\cdots\times M_k$ means $H^*(M;\mathbb{Q})$ cannot be singly generated.
\end{proof}
Coupling the Berger classification with the de Rham decomposition, one obtains a classification of reducible holonomy groups by requiring that each factor is one of the examples from Berger's list, realized by compact simply connected manifolds. One can find metrics on product manifolds with irreducible holonomy so the converse of Theorem \ref{holonomytheorem} is false. More precisely, a generic manifold, with $O(n)$ holonomy or $SO(n)$ if it is orientable, will have irreducible holonomy and have homology that is not generated by a single element.
\begin{corollary}
If $M$ is oriented then $\text{Hol}(M)=SO(n)$ so there are at least two non-trivial classes in $H^*(M;\mathbb{Q})$. Thus, if $\text{Hol}(M)=\prod_{i=1}^k \text{Hol}(M_i)$ then $H^*(M;\mathbb{Q})$ must have at least two generators.
\end{corollary}

Since Ziller \cite{ziller} showed $b_k(\Lambda M;\mathbb{F}_2)$ to be unbounded for $M$ a compact globally symmetric space of rank $>1$, we only consider simply connected Riemannian manifolds which are irreducible and non-symmetric, as well as non-compact symmetric spaces of rank $1$. We obtain the following classification of reducible holonomy groups from Berger's classification and de Rham's decomposition theorem. We require that each factor in the decomposition is one from Berger's list or the holonomy of an irreducible non-compact symmetric space of rank $1$, which may be interchanged with a compact symmetric space of rank $1$ under bijection in the decomposition.
Let $N$ be a simply connected Riemannian symmetric space. Then it may decomposed as $N=N_0\times N_1\times\cdots\times N_r$ for $N_0$ Euclidean and the other $N_i$ irreducible. As such, the classification of symmetric spaces is reduced to the the irreducible case. By duality de Rham (see \cite{helgason}, P. 199), 
compact and noncompact simply connected irreducible symmetric spaces may be interchanged, which therefore simplifies the classification to that of compact irreducible symmetric spaces. 

Let $M$ be an $n$-dimensional simply connected Riemannian manifold that is not a compact symmetric space of rank $>1$. The holonomy is then a direct product decomposition of: 
\begin{equation*}
\begin{split}
\text{Hol}(M)=&\prod_{i_1=1}^{N_1}SO(n_1)\prod_{i_2=1}^{N_2} U(n_2)\prod_{i_1=3}^{N_3} SU(n_3)\prod_{i_4=1}^{N_4} (Sp(n_4)\cdot Sp(1))\prod_{i_5=1}^{N_5} Sp(n_5)\prod_{i_6=1}^{N_6} G_2\prod_{i_7=1}^{N_7} Spin(7)\\ &\times \text{Hol}(\text{compact symmetric spaces of rank}= 1)
\end{split}
\end{equation*}
where it is required that a given decomposition has a least two irreducible factors and $1\le N_k<\infty$ for every $k$ in the expansion. Adopting the Cartan labeling convention, we consider the label AIII, BDI, CII, FII symmetric spaces $G/K$.
\begin{remark}
For label AII, $G/K=SU(p+q)/S(U(p)\times U(q))$ has $\text{rank}(G/K)=\min(p,q)$ and $\dim(G/K)=2pq$. So we let $p=1$, such that $K=S(U(1)\times U(q))$. If we let $\dim(G/K)=n=2q$ then $q=\frac{n}{2}$ or $K=S\left(U(1)\times U\big(\frac{n}{2}\big)\right)$.
\end{remark}
\begin{remark}
For label BDI, $G/K=SO(p+q)/SO(p)\times SO(q)$ has $\text{rank}(G/K)=\min(p,q)$ and  $\dim(G/K)=pq$. So if we let $p=1$, $SO(1+q)/SO(1)\times SO(q)$ has dimension $q$. Thus, $K=SO(1)\times SO(n)$ for $\dim(G/K)=n$.
\end{remark}
\begin{remark}
For label CII, $G/K=Sp(p+q)/Sp(p)\times Sp(q)$ has $\text{rank}(G/K)=\min(p,q)$ and $\dim (G/K)=4pq$. We let $p=1$ so $F_4/Spin(9)$ has dimension $n=4q$ or $K=Sp(1)\times Sp\left(\frac{n}{4}\right)$.
\end{remark}
\begin{remark}
For label FII, $G/K=F_4/Spin(9)$ has $\text{rank}(G/K)=1$ and $\dim(G/K)=16$.
\end{remark}
Note, if $M=\mathbb{R}^n$ with the flat Euclidean metric then parallel translation is simply a translation in $\mathbb{R}^n$. In particular, if $P_{\gamma} : T_{\gamma(0)}M\to T_{\gamma(1)}M$ for $\gamma:[0,1]\to M$, then $P_{\gamma}$ is the identity for each $\gamma$ so $\text{Hol}(\mathbb{R}^n)$ is the trivial group, and it may be ignored in the decomposition. Thus, the decomposition of compact symmetric spaces for the problem under consideration is $\text{Hol(compact simply connected symmetric spaces)}=S\left(U(1)\times U\big(\frac{n}{2}\big)\right)\times (SO(1)\times SO(n))\times \left(Sp(1)\times Sp\big(\frac{n}{4}\big)\right)\times Spin(9).$
To be more precise, the holonomy of a reducible simply connected Riemannian manifold $M$ (excluding compact symmetric spaces of rank $>1$) is given by the following.
\begin{theorem} (Classification of Reducible Holonomy.) 
If $M$ is an $n$-dimensional simply connected Riemannian manifold with decomposition $M=M_1\times\cdots\times M_j$ of irreducible $M_k$ ($k=1,\dots, j)$, then it has holonomy decomposition given by the direct product:
\begin{equation*}
\begin{split}
\text{Hol}(M)=&\prod_{i_1=1}^{N_1}SO(n_1)\prod_{i_2=1}^{N_2} U(n_2)\prod_{i_3=1}^{N_3} SU(n_3)\prod_{i_4=1}^{N_4}(Sp(n_4)\cdot Sp(1))\prod_{i_5=1}^{N_5} Sp(n_5)\prod_{i_6=1}^{N_6} G_2\prod_{i_7=1}^{N_7} Spin(7)\\ &\prod_{i_8=1}^{N_8} S(U(1)\times U(n_6)) \prod_{i_9=1}^{N_9} (SO(1)\times SO(n_7))\prod_{i_{10}=1}^{N_{10}} (Sp(1)\times Sp(n_8))\prod_{i_{11}=1}^{N_{11}} Spin(9)
\end{split}
\end{equation*} where it is required that there are at least two factors in a given decomposition and $\sum_i\dim_{\mathbb{R}}(M_i)=n$.
\end{theorem}
\begin{corollary}
If $M$ is a simply connected Riemannian manifold with reducible holonomy given by the previous theorem, the rational cohomology ring $H^*(M;\mathbb{Q})$ has at least two generators.
\end{corollary}
Since Riemannian spaces that are locally isometric to the homogeneous spaces $G/H$ have local holonomy isomorphic to $H$, the holonomy decomposition is
\begin{equation*}
\begin{split}
\text{Hol}(M)=&\prod_{i_1=1}^{N_1}SO(n_1)\prod_{i_2=1}^{N_2} U(n_2)\prod_{i_3=1}^{N_3} SU(n_3)\prod_{i_4=1}^{N_4} (Sp(n_4)\cdot Sp(1))\prod_{i_5=1}^{N_5} Sp(n_5)\prod_{i_6=1}^{N_6} G_2\prod_{i_7=1}^{N_7} Spin(7)\\ & \prod_{i_8=1}^{N_8} \text{Hol}\left(\frac{SU(n_6+1)}{S(U(1)\times U(n_6))}\right) \prod_{i_9=1}^{N_9} \text{Hol}\left(\frac{SO(n_7+1)}{SO(1)\times SO(n_7)}\right)\prod_{i_{10}=1}^{N_{10}} \text{Hol}\left(\frac{Sp(n_8+1)}{Sp(1)\times Sp(n_8)}\right)\\ & \prod_{i_{11}=1}^{N_{11}} \text{Hol}\left(\frac{F_4}{Spin(9)}\right).
\end{split}
\end{equation*}

Let $M$ be any space whose rational cohomology ring is a free graded-commutative algebra. Then the rationalization $M_{\mathbb{Q}}$ is a product of Eilenberg-MacLane spaces. This hypothesis on cohomology also applies to compact Lie groups. Consider the decomposition of $M$, i.e. $M=M_1\times\cdots\times M_k$, then the rationalization of holonomy is $\text{Hol}(M)_{\mathbb{Q}}=\text{Hol}(M)\times_{\mathbb{Q}}E\mathbb{Q}$ and $\text{Hol}(M_{\mathbb{Q}})=\text{Hol}(M\times_{\mathbb{Q}}E\mathbb{Q})=\text{Hol}(M)\times \text{Hol}(E\mathbb{Q})$. The total space of the universal bundle over $B\mathbb{Q}$ is $E\mathbb{Q}=B\mathbb{Q}\rtimes\mathbb{Q}$ since $B\mathbb{Q}=E\mathbb{Q}/\mathbb{Q}$. Note, $B\mathbb{Q}$ is the Eilenberg-MacLane space $K(\mathbb{Q},1)$ so $E\mathbb{Q}=K(\mathbb{Q},1)\rtimes\mathbb{Q}$.

We can obtain the space $K(\mathbb{Q},1)$ by the following procedure. Take the circle $S^1$, consider the sequence of maps $f_n:S^1\to S^1$ of degree $n$, and form an infinite mapping telescope of $f_n$, a special case of a homotopy colimit. Observe, $\mathbb{Q}$ is the filtered colimit $\mathbb{Z}\to\mathbb{Z}\to\cdots$ where successive maps are multiplication by $1,2,\dots$. Note, since $\Omega\mathbb{Q}\cong\ast$ regardless of its topology, the resulting classifying space $B\mathbb{Q}$ will be a $K(\mathbb{Q},1)$ space. It follows that $\text{Hol}(M)_{\mathbb{Q}}=\text{Hol}(M)\times_{\mathbb{Q}}E\mathbb{Q}$ and $\text{Hol}(M_{\mathbb{Q}})=\text{Hol}(M)\times \text{Hol}(E\mathbb{Q})$, so $\text{Hol}(M)_{\mathbb{Q}}=\frac{\text{Hol}(M_{\mathbb{Q}})}{\text{Hol}(E\mathbb{Q})}\times_{\mathbb{Q}}E\mathbb{Q}$. By first computing the rationalization space $M_{\mathbb{Q}}$, we can determine the equivariant cohomology ring $H^*_{\mathbb{Q}}(M;\mathbb{Q})=H^*(E\mathbb{Q}\times_{\mathbb{Q}}M;\mathbb{Q})$, noting that the rationalization $M_{\mathbb{Q}}$ is the product of Eilenberg-MacLane spaces whose rational cohomology rings are easier to compute. Since the decomposition of $M$ is given by:
\begin{equation*}
\begin{split}
M=&\prod_{i_1=1}^{N_1}M_{\text{orientable, }SO(n_1)}\prod_{i_2=1}^{N_2} M_{\text{K\"{a}hler, } U(n_2)}\prod_{i_3=1}^{N_3} M_{\text{Calabi-Yau, }SU(n_3)}\prod_{i_4=1}^{N_4} M_{\text{quaternion-K\"{a}hler, }Sp(n_4)\cdot Sp(1)}\\
& \prod_{i_5=1}^{N_5} M_{\text{hyperk\"{a}hler, }Sp(n_5)}\prod_{i_6=1}^{N_6} M_{G_2}\prod_{i_7=1}^{N_7} M_{Spin(7)}\prod_{i_8=1}^{N_8}\frac{SU(n_6+1)}{S(U(1)\times U(n_6))}\prod_{i_9=1}^{N_9}\frac{SO(n_7+1)}{SO(1)\times SO(n_7)}\\
&\prod_{i_{10}=1}^{N_{10}} \frac{Sp(n_8+1)}{Sp(1)\times Sp(n_8)}\prod_{i_{11}=1}^{N_{11}} \frac{F_4}{Spin(9)},
\end{split}
\end{equation*} we can compute the rational cohomology ring $H^*(M;\mathbb{Q})$ by the Kunneth formula.
\begin{equation*}
\begin{split}
H^*(M;\mathbb{Q})\cong & \bigotimes_{i_1=1}^{N_1}H^*(M_{\text{orientable, }SO(n_1)};\mathbb{Q})\bigotimes_{i_2=1}^{N_2} H^*(M_{\text{K\"{a}hler, } U(n_2)};\mathbb{Q}) \bigotimes_{i_3=1}^{N_1} H^*(M_{\text{Calabi-Yau, }SU(n_3)};\mathbb{Q})\\ &\otimes\cdots\otimes\bigotimes_{i_{11}=1}^{N_{11}} H^*(F_4/Spin(9);\mathbb{Q})
\end{split}
\end{equation*}
\section{Palais-Smale Local Approximation for Free Loop Space}
Let $\phi$ be a chart on the Hilbert manifold $\Lambda M$ defined by 
\begin{equation}
\begin{split}
\phi:\quad& \mathcal{U}\to \mathcal{V}\subset\mathbb{R}^n \\ 
& \gamma\mapsto  \left(\int_{S^1}\|\dot\gamma(t)\|^2dt,\int_{S^1}\|\dot\gamma(t)\|^2dt,\dots \right),
\end{split}
\end{equation} satisfying the following commutative diagram:
\newline
\[
\begin{tikzcd}[column sep=large, row sep=large]
    \mathcal{U}
     \arrow[swap]{d}{E}
 \arrow{r}{\phi}
& \mathcal{V}\subset\mathbb{R}^n \\
  \mathbb{R} \arrow[swap]{ru}{\phi\circ E^{-1}}
&
\end{tikzcd}
\]
The Morse Lemma can be generalized to the infinite-dimensional setting. On separable Hilbert spaces, it takes the following form.
\begin{theorem}\label{generalizedmorse}
Let $H$ be a separable Hilbert space and $f:H\to\mathbb{R}$ a $C^k$ function $f$ with $k\ge 3$ in the sense of Fr\'echet differentiability, for which $c$ is a non-degenerate critical point. Then there exist convex neighborhoods $\mathcal{U}$ and $\mathcal{V}$ of $c$, and a diffeomorphism $\varphi:\mathcal{U}\to\mathcal{V}$ of class $C^{k-2}$ with $\varphi(c)=c$ and a bounded orthogonal projection $P: H\to H$ such that 
\begin{equation}
f(x)=f(c)-\|P(\varphi(x))\|_H^2+\|\varphi(x)-P(\varphi(x))\|_H^2,
\end{equation} where $\dim\text{im }P$ is the Morse index of the critical point $c$.
\end{theorem}
\begin{theorem} \label{result}
If the energy functional $E:\Lambda M\to\mathbb{R}$, defined by $E[\gamma]=\int_{S^1}\|\dot\gamma(t)\|^2dt$ for $\gamma:S^1\to M$, is a $C^3$-function, then $M$ is a union of infinitely many Baire spaces $S^{\infty}$. 
\end{theorem}
\begin{proof}
Assume that the Riemannian metric on $M$ is chosen such that the energy functional $E:\Lambda M\to\mathbb{R}$ is Morse-Bott, i.e., $E$ is a $C^2$-function. Assume that $E$ is also a $C^3$-function by choice. Let $H_{\alpha}\subset\Lambda M$ and let $\Lambda M=\bigcup_{\alpha\in A}H_{\alpha}\oplus H_{\alpha}^{\perp}$ for $A$ the set of critical points of $E$. Further, let ${\mathcal{U}}_{\alpha}\subset\Lambda M$ be a local neighborhood of the critical point $\beta_{\alpha}$. Then the projection map $P:\mathcal{U}_{\alpha}\to H_{\alpha}$ has $P(x)=m$ for $x=m+m',m\in H_{\alpha},m'\in H_{\alpha}^{\perp}$. Recall that $\dim\text{im }P=\lambda(\beta_{\alpha})$. Let $\{\gamma_{\alpha k}\}_{k=1}^{\lambda(\beta_{\alpha})}$ be an orthonormal basis for $H_{\alpha}$. That is, the orthonormal basis satisfies the following criteria.
\begin{enumerate}[label=(\roman*)]
\item Orthogonality: Every two different basis elements are orthogonal, i.e. $\langle\gamma_k,\gamma_j\rangle=0$ for all $k,j=1,\dots,\lambda(\beta_{\alpha})$ with $k\ne j$.
\item Normalization: Every element of this basis has norm $1$, i.e. $\|\gamma_{\alpha k}\|=1$ for all $k=1,\dots,\lambda(\beta_{\alpha})$.
\item Completeness: The linear span of the family $\gamma_{\alpha k}$, for $k=1,\dots,\lambda(\beta_{\alpha})$, is dense in $H_{\alpha}$.
\end{enumerate} 
Similarly, let $\{\gamma_{\alpha k}^{\perp}\}_{k\in B}$, for $B:=(\lambda(\beta_{\alpha})+1,\lambda(\beta_{\alpha})+2,\dots)$ a countable but infinite set, be an orthonormal basis for $H_{\alpha}^{\perp}$, satisfying the above properties, of $\dim H_{\alpha}^{\perp}=\text{codim }H_{\alpha}=\infty$ for $\dim H_{\alpha}=\lambda(\beta_{\alpha})$. Then $\langle\gamma_{\alpha k},\gamma_{\alpha j}^{\perp}\rangle=0$ for all $k=1,\dots,\lambda(\beta_{\alpha}),j=\lambda(\beta_{\alpha})+1,\dots$. We may construct $\Lambda M$ as a patching over neighborhoods for sufficiently many critical points, which are homeomorphic to $H_{\alpha}\oplus H_{\alpha}^{\perp}$ for $\alpha\in A$. Thus, for $\mathcal{U}_{\alpha}$ an open neighborhood of $\beta_{\alpha}\in\Lambda M$, $\mathcal{U}_{\alpha}=H_{\alpha}\oplus H_{\alpha}^{\perp}$ such that for $\gamma_{\alpha}\in\mathcal{U}_{\alpha}$, there exists a unique $\theta_{\alpha}\in H_{\alpha}$ and $\varepsilon_{\alpha}\in H_{\alpha}^{\perp}$ such that $\gamma_{\alpha}=\theta_{\alpha}+\varepsilon_{\alpha}$ whereby $\theta_{\alpha}=\sum_{k=1}^{\lambda(\beta_{\alpha})}c_{\alpha k}\gamma_{\alpha k}$ and $\varepsilon_{\alpha}=\sum_{j=\lambda(\beta_{\alpha})+1}^{\infty}\widehat{c_{\alpha j}}\gamma_{\alpha j}^{\perp}$. Thus,
\begin{equation}
\gamma_{\alpha}=\sum_{k=1}^{\lambda(\beta_{\alpha})}c_{\alpha k}\gamma_{\alpha k}+\sum_{j=\lambda(\beta_{\alpha})+1}^{\infty}\widehat{c_{\alpha j}}\gamma_{\alpha j}^{\perp}
\end{equation} where $\langle\gamma_{\alpha k},\gamma_{\alpha j}^{\perp}\rangle=0$. We invoke Theorem \ref{approximation} to conclude that $E:\Lambda M\to\mathbb{R}$ may be approximated by a smooth functional $E_{\varepsilon}:\Lambda M'\to\mathbb{R}$ with infinitely many critical points so that $A$ in $\Lambda M=\bigcup_{\alpha\in A}H_{\alpha}\oplus H_{\alpha}^{\perp}$ is uncountable or countably infinite. Note, this follows because $\mathcal{U}_{\alpha}$ is dense. 

Furthermore, let $\{\gamma_{\alpha k}\}_{k=1}^{\infty}$ be an orthonormal basis for $\mathcal{U}_{\alpha}=H_{\alpha}\oplus H_{\alpha}^{\perp}$ such that if $\gamma_{\alpha}\in\mathcal{U}_{\alpha}$ then $\gamma_{\alpha}=\sum_{k=1}^{\infty}c_{\alpha k}\gamma_{\alpha k}$ and $P:\mathcal{U}_{\alpha}\to H_{\alpha}$ means $P(\gamma_{\alpha})=P\left(\sum_{k=1}^{\infty}c_{\alpha k}\gamma_{\alpha k}\right)$. Since $P$ is a linear operator, $P(\gamma_{\alpha})=\theta_{\alpha}=\sum_{k=1}^{\infty}c_{\alpha k}P(\gamma_{\alpha k})$. Let $\{\theta_{\alpha k}\}_{k=1}^{\infty}$ be an orthonormal basis for $H_{\alpha}$ where $\theta_{\alpha k}=0$ for $k>\lambda(\beta_{\alpha})$. That is, $\{\theta_{\alpha k}\}_{k=1}^{\lambda(\beta_{\alpha})}$ is an orthonormal basis for $H_{\alpha}$. Then $P(\gamma_{\alpha k})=\theta_{\alpha k}$ and $P(\gamma_{\alpha k})=\theta_{\alpha k}=0$ for $k>\lambda(\beta_{\alpha})$. Let $\{\theta_{\alpha k}^{\perp}\}_{k=\lambda(\beta_{\alpha})+1}^{\infty}$ be an orthonormal basis for $H_{\alpha}^{\perp}$. Then for $P_{\perp}:\mathcal{U}_{\alpha}\to H_{\alpha}^{\perp}$ and $P_{\perp}(\gamma_{\alpha k})=\theta_{\alpha k}^{\perp}$, we have $P_{\perp}(\gamma_{\alpha})=\varepsilon_{\alpha}=\sum_{k=1}^{\infty}\widehat{c_{\alpha k}}P_{\perp}(\gamma_{\alpha k})$ such that $P_{\perp}(\gamma_{\alpha k})=\theta_{\alpha k}^{\perp}=0$ for $k<\lambda(\beta_{\alpha})+1$. Therefore $\gamma_{\alpha}=\sum_{k=1}^{\infty}c_{\alpha k}\gamma_{\alpha k}$ and 
\begin{equation*}
\begin{split}
P(\gamma_{\alpha})=\theta_{\alpha}&=\sum_{k=1}^{\infty}c_{\alpha k}P(\gamma_{\alpha k})\\&=\sum_{k=1}^{\lambda(\beta_{\alpha})}c_{\alpha k}P(\gamma_{\alpha k}) \text{ with }
P(\gamma_{\alpha k})=\begin{cases} 
      \theta_{\alpha k}, & \text{if }k\le\lambda(\beta_{\alpha}), \\
      0, & \text{if }k>\lambda(\beta_{\alpha}), 
         \end{cases} 
  \end{split}
  \end{equation*}
\begin{equation*}
\begin{split}
P_{\perp}(\gamma_{\alpha})=\varepsilon_{\alpha}&=\sum_{k=1}^{\infty}\widehat{c_{\alpha k}}P_{\perp}(\gamma_{\alpha k})\\&=\sum_{k=\lambda(\beta_{\alpha})+1}^{\infty}\widehat{c_{\alpha k}}P_{\perp}(\gamma_{\alpha k})\text{ with } P_{\perp}(\gamma_{\alpha k})=\begin{cases} 
      \theta_{\alpha k}^{\perp}, & \text{if }k\ge\lambda(\beta_{\alpha})+1, \\
      0, & \text{if }k<\lambda(\beta_{\alpha})+1, 
         \end{cases} 
\end{split}
\end{equation*} where $\langle\theta_{\alpha k},\theta_{\alpha j}^{\perp}\rangle=0$ for all $k=1,\dots,\lambda(\beta_{\alpha}), j=\lambda(\beta_{\alpha})+1,\dots.$ Thus, by the Morse-Palais Lemma \ref{generalizedmorse}, $E[\gamma_{\alpha}]=E[\beta_{\alpha}]-\|P(\varphi(\gamma_{\alpha}))\|^2_{\Lambda M}+\|\varphi(\gamma_{\alpha})-P(\varphi(\gamma_{\alpha}))\|^2_{\Lambda M}$ where $\varphi:\mathcal{U}\to\mathcal{V}$ is a diffeomorphism such that for $\gamma_{\alpha}\in\mathcal{U}_{\alpha}$, $\varphi(\gamma_{\alpha})=\gamma_{\alpha}$, i.e. $\varphi=\text{id}_{\mathcal{U}_{\alpha}}$ by identification. It follows that 
\begin{equation}
\begin{split}
&E[\gamma_{\alpha}]=E[\beta_{\alpha}]-\|P(\varphi(\gamma_{\alpha}))\|^2_{\Lambda M}+\|\varphi(\gamma_{\alpha})-P(\varphi(\gamma_{\alpha}))\|^2_{\Lambda M} \\
&=E[\beta_{\alpha}]-\|P(\varphi(\gamma_{\alpha}))\|^2_{\Lambda M}+\|\gamma_{\alpha}\|^2_{\Lambda M}-2Re\langle\gamma_{\alpha},P(\gamma_{\alpha})\rangle+\|P(\gamma_{\alpha})\|^2_{\Lambda M} \\ &=E[\beta_{\alpha}]+\|\gamma_{\alpha}\|^2_{\Lambda M}-2Re\langle\gamma_{\alpha},P(\gamma_{\alpha})\rangle.
\end{split}
\end{equation} 
Let \begin{equation*}
\begin{split}
\varphi_{\alpha}:\quad &\mathcal{U}_{\alpha}\cap H_{\alpha}\to\mathbb{R}^{\lambda(\beta_{\alpha})} \\
&\gamma_{\alpha}\mapsto \left(\gamma_{\alpha 1},\dots,\gamma_{\alpha \lambda(\beta_{\alpha})}\right)
\end{split}
\end{equation*} since $H_{\alpha}$ is a vector space, so the tangent space $T_{\gamma_{\alpha}}H_{\alpha}$ at any point $\gamma_{\alpha}\in H_{\alpha}$ is canonically isomorphic to $H_{\alpha}$ itself. Similarly, let
\begin{equation*}
\begin{split}
\widehat{\phi_{\alpha}}:\quad &\mathcal{U}_{\alpha}\cap H_{\alpha}^{\perp}\to\mathbb{R}^{\infty}\\
&\gamma_{\alpha}\mapsto(\gamma_{\alpha 1},\gamma_{\alpha 2},\dots)
\end{split}
\end{equation*} with $H_{\alpha}\cap H_{\alpha}^{\perp}=\{\beta_{\alpha}\}$, affinely translated $\{0\}$. Then let 
\begin{equation*}
\begin{split}
\phi_{\alpha}:\quad&\mathcal{U}_{\alpha}\to\mathbb{R}^{\infty} \\
& \gamma_{\alpha}\mapsto (\gamma_{\alpha 1},\gamma_{\alpha 2},\dots)
\end{split}
\end{equation*} be a chart on $\mathcal{U}_{\alpha}$. Thus, 
\begin{equation*}
\begin{split}
\phi_{\alpha}^i:\quad&\mathcal{U}_{\alpha}\to\mathbb{R}\\
&\gamma_{\alpha}\mapsto \gamma_{\alpha i}
\end{split}
\end{equation*}
are local coordinate charts such that $\gamma_{\alpha}=\sum_{k=1}^{\infty}c_{\alpha k}\gamma_{\alpha k}=\sum_{k=1}^{\infty}c_{\alpha k}\phi_{\alpha}^k$ so $\|\gamma_{\alpha}\|^2_{\Lambda M}=\langle \gamma_{\alpha},\gamma_{\alpha}\rangle_{\Lambda M}=\left\|\sum_{k=1}^{\infty}c_{\alpha k}\phi_{\alpha}^k\right\|^2_{\Lambda M}=\sum_{k=1}^{\infty}\left\|c_{\alpha k}\phi_{\alpha}^k\right\|^2_{\Lambda M}$ since the $\phi_{\alpha}^k$'s are orthogonal by choice, and $\|\gamma_{\alpha}\|^2_{\Lambda M}=\sum_{k=1}^{\infty}c_{\alpha k}^2\|\phi_{\alpha}^k\|^2_{\Lambda M}$ where $\|\phi_{\alpha}^k\|^2_{\Lambda M}=\|\phi_{\alpha}^k\|^2_0+\|\nabla_t\phi_{\alpha}^k\|^2_0$. It follows that 
\begin{equation*}
\|\gamma_{\alpha}\|^2_{\Lambda M}=\sum_{k=1}^{\infty}c_{\alpha k}^2\left(\|\phi_{\alpha}^k\|^2_0+\|\nabla_t\phi_{\alpha}^k\|^2_0\right).
\end{equation*}
Likewise, $\theta_{\alpha}=\sum_{k=1}^{\lambda(\beta_{\alpha})}\widehat{c_{\alpha k}}P(\phi_{\alpha}^k)=\sum_{k=1}^{\lambda(\beta_{\alpha})}\widehat{c_{\alpha k}}\phi_{\alpha}^k$ since $P(\phi_{\alpha}^k)=\phi_{\alpha}^k$ for $k\le\lambda(\beta_{\alpha})$. Recall $\gamma_{\alpha}=\sum_{k=1}^{\infty}c_{\alpha k}\gamma_{\alpha k}$. Therefore, $\langle \gamma_{\alpha}, P(\gamma_{\alpha})\rangle_{\Lambda M}=\left\langle\sum_{k=1}^{\infty}c_{\alpha k}\phi_{\alpha}^k, \sum_{j=1}^{\lambda(\beta_{\alpha})}\widehat{c_{\alpha j}}\phi_{\alpha}^j\right\rangle_{\Lambda M}=\sum_{(k,j)\in K\times J}c_{\alpha k}\widehat{c_{\alpha j}}\langle\phi_{\alpha}^k, \phi_{\alpha}^j\rangle=c_{\alpha k}\widehat{c_{\alpha k}}:=C_{\alpha}$. As such, $E[\gamma_{\alpha}]=E[\beta_{\alpha}]+\sum_{k=1}^{\infty}c_{\alpha k}^2\|\phi_{\alpha}^k\|^2_{\Lambda M}-2C_{\alpha}$ or, for $\widehat{C_{\alpha}}=E[\gamma_{\alpha}]-E[\beta_{\alpha}]+2C_{\alpha}$, we have $\sum_{k=1}^{\infty}c_{\alpha k}^2\|\phi_{\alpha}^k\|^2_{\Lambda M}=\widehat{C_{\alpha}}$ where $\gamma_{\alpha}=\sum_{k=1}^{\infty}c_{\alpha k}\phi_{\alpha}^k=c_{\alpha 1}\phi_{\alpha}^1+c_{\alpha 2}\phi_{\alpha}^2+\cdots$, and $\phi_{\alpha}^k$ is defined on $\mathcal{U}_{\alpha}$. Then $c_{\alpha k}=\int_{\mathcal{U}_{\alpha}}\phi_{\alpha}^k(\gamma_{\alpha})d\gamma_{\alpha}$. Thus, $\sum_{k=1}^{\infty}\frac{C_{\alpha k}^2}{\widehat{C}_{\alpha}}\|\phi_{\alpha}^k\|^2_{\Lambda M}=1$ so letting $a_{\alpha k}=C^2_{\alpha k}/\widehat{C_{\alpha}}$ we have
\begin{equation}
\sum_{k=1}^{\infty}a_{\alpha k}\|\phi_{\alpha}^k\|^2_{\Lambda M}=1.
\end{equation} The locus of all points $\{\phi_{\alpha}^1,\phi_{\alpha}^2,\dots\}$, i.e. the coordinates on the patch $\mathcal{U}_{\alpha}$, corresponds to an infinite-dimensional ellipsoid with $a_{\alpha i}\ne 0$ for $k\in\mathbb{Z}^+$ because, otherwise, $a_{\alpha i}= 0$ would obstruct the uniqueness and existence of the orthonormal basis. Thus, $\sum_{k=1}^{\infty}a_{\alpha k}\|\phi_{\alpha}^k\|^2_{\Lambda M}=1$ for each $\mathcal{U}_{\alpha}$. It follows that
\begin{equation}
\Lambda M=\bigcup_{\alpha\in A}\mathcal{U}_{\alpha}=\bigcup_{\alpha\in A}\left\{(\phi_{\alpha}^1,\phi_{\alpha}^2,\dots)\in\mathbb{R}^{\infty}\Bigg | \sum_{k=1}^{\infty}a_{\alpha k}\|\phi_{\alpha}^k\|^2_{\Lambda M}=1\right\}\cong \bigcup_{\alpha\in A}S^{\infty}.
\end{equation} Suppose there exists a suitable $C^0$-perturbation $E_{\varepsilon}$ of $E$ such that there is an atlas with $\Lambda M=\bigcup_{\alpha\in A}\mathcal{U}_{\alpha}$ where $\mathcal{U}_{\alpha}$ is locally a Hilbert space for all $\alpha\in A$. Then given such a distribution of critical points on $\Lambda M$, we can compute the cohomology groups via a nerve of a good cover and, thus, the sequence of Betti numbers. Note, such a construction of an atlas for $\Lambda M$ can always be made because $E_{\varepsilon}$ is an arbitrary $C^0$-perturbation of $E$. Then the rational cohomology of the Baire space $H^k(S^{\infty};\mathbb{Q})$ is the inverse limit $\varprojlim_nH^k(S^{n};\mathbb{Q})$.
\end{proof}
\begin{lemma}
Let $\Lambda M\subset\mathbb{R}^{\infty}$ be a paracompact connected Hilbert manifold. Moreover, let $E:\Lambda M\to\mathbb{R}$ be a $C^2$-function. Then the simplicial complex $\Delta:=\{\gamma\in \Lambda M:dE[\gamma]=0\}$, for vertices identified with the intersections of a refined open cover $\{U_{\alpha}\}_{\alpha\in A}$ the non-degenerate critical points of $E$, is homotopy equivalent to $\Lambda M$.
\end{lemma}
\begin{proof}
A $C^0$-perturbation of the energy functional $E:\Lambda M\to\mathbb{R}$ produces infinitely many critical points. Suppose we form a good cover $\mathcal{U}=\{U_{\alpha_k}\}_{\alpha_k\in A}$ such that each $U_{\alpha_k}$ is a neighborhood of the non-degenerate critical point ${\gamma}_k$. Then, in general, this should be a cover of $\Lambda M$ and the map from the geometric realization of the simplicial space, i.e. the topological nerve of $\{U_{a_k}\}$, to $\Lambda M$ is a homotopy equivalence. Likewise, the geometric realization of the topological nerve $R(\text{Top-Nrv}(\mathcal{U}))$ is homotopy equivalent to the geometric realization of the simplicial set for which the $|A|$-simplices are $|A|$-tuples $(\alpha_0,\dots,\alpha_{|A|})$, whereby the intersection $\bigcap_{k\in\mathbb{Z}}U_{\alpha_k}$ is nonempty and $|A|\to\infty$ under $C^0$-perturbation of $E$.
\end{proof}

Treating $\Lambda M$ as a CW complex, by Theorem \ref{result}, it is given by the union of Baire spaces, namely $M=\bigcup_{\alpha\in A}S^{\infty}$ without information on intersections. The only requirement is that the union of such Baire spaces covers $\Lambda M$. Let $\Sigma=\{S^{\infty},\dots,S^{\infty}\}$ be a finite collection of sets, with cardinality $|A|$. The nerve consists of all sub-collections whose sets have a non-empty common intersection, $\text{Nrv}(\Sigma)=\left\{X\subseteq\Sigma:\bigcap X\ne\emptyset\right\}$, which is an abstract simplicial complex. That is, suppose we form a good cover $\mathcal{U}=\{U_{\alpha_k}\}_{\alpha_k\in A}$ such that each $U_{\alpha_k}$ is a neighborhood of the non-degenerate critical point $\gamma_k$. Since all Hilbert manifolds are paracompact, we may compute homology of the CW model of $\Lambda M$ via \v{C}ech homology:
\begin{equation*} \check{H}_k(\mathcal{U},\mathcal{F})=\varinjlim_{\mathcal{U}\in A}H_k(\mathcal{U},\mathcal{F})
\end{equation*} for $\mathcal{F}$ a presheaf of abelian groups on $\Lambda M$, and $A$ a directed set consisting of all open covers of $\Lambda M$, which is directed by  $\mathcal{U}<\mathcal{V}$ for $\mathcal{V}$ a refinement of $\mathcal{U}$. \v{C}ech homology is similarly given by $\check{H}_k(\Lambda M,\mathcal{F})=\varinjlim_{\mathcal{U}\in A}H_k(\text{Nrv}(\mathcal{U}),\mathcal{F})$.

\section{Integral Cellular Homology and Smith Normal Form for Free Loop Space\label{sec: SNF}}
\begin{lemma}
Every continuous map $f:M\to\mathbb{R}^n$ from a Hilbert manifold can be arbitrarily approximated by a smooth map $g:M\to\mathbb{R}^n$ which has no critical points.
\end{lemma}
\begin{theorem}\label{approximation}
A continuous map $f:M\to\mathbb{R}^n$ from a Hilbert manifold may be approximated to create infinitely many local minima and maxima by a small $C^0$-perturbation. This is impossible, in general, for a $C^1$-perturbation.
\end{theorem}
\begin{proof}
Working in local charts, let $\sigma(t):=\max (0,\min(2t,1))$ for $t\in\mathbb{R}$. For $\varepsilon>0$, define $f_{\varepsilon}(x)=f(\sigma(\|x\|/\varepsilon)x)$. Then $f_{\varepsilon}$ is constant in the ball of radius $\varepsilon/2$, it coincides with $f$ outside the ball of radius $\varepsilon$, and $\|f-f_{\varepsilon}\|_{\infty}=o(1)$ for $\varepsilon\to 0$.
\end{proof}

\begin{lemma}[Homotopy Invariance]\label{invariance}
The respective homotopy types of a closed Riemannian manifold $M$ and its (free) loop space $\Lambda M$ are invariant under a $C^0$-perturbation of the energy functional $E:\Lambda M\to\mathbb{R}$ in an open neighborhood of a critical point $\gamma\in\Lambda M$ of index $\lambda$. For $\varepsilon>0$, such a local $C^0$-perturbation of $E$ is given by $E_{\varepsilon}[\gamma]=E[\sigma(\|\gamma\|/\varepsilon)\gamma]$ where $\sigma(t):=\max(0,\min(2t,1))$ for $t\in\mathbb{R}$.
\end{lemma}
\begin{proof}
The canonical fibration $\Omega M\hookrightarrow\Lambda M\overset{\text{ev}}\longrightarrow M$ admits the section $M\hookrightarrow\Lambda M$ given by the inclusion of constant loops, which implies that the associated long exact sequence of homotopy reduces to the split short exact sequences $0\to\pi_k(\Omega M)\to\pi_k(\Lambda M)\to\pi_k(M)\to 0$. Thus, for $k\ge 1$ we have canonical isomorphisms
\begin{equation*}
\pi_k(\Lambda M)\cong\pi_k(\Omega M)\rtimes\pi_k(M)\cong\pi_{k+1}(M)\rtimes\pi_k(M).
\end{equation*} Then for $\Lambda M=E^{-1}(-\infty,\infty)$, define by $\Lambda M':=E^{-1}_{\varepsilon}(-\infty,\infty)$ the free loop space corresponding to the perturbed $M$, say $M'$, after a $C^0$-perturbation $E_{\varepsilon}$ of the energy functional. For a $C^0$-perturbation, $E_{\varepsilon}$ is constant in the ball of radius $\varepsilon/2$ and it coincides with $E$ outside the ball of radius $\varepsilon$. The point $\gamma\in\Lambda M$ has index $\lambda$. Then let $\alpha$ be the number of critical points, excluding $\gamma$, of index $\lambda$. All of the infinitely many local minima or maxima in the $\varepsilon$-neighborhood created by the $C^0$-perturbation are also of index $\lambda$. Therefore, the number of critical points of index $\lambda$ changes from $1+\alpha$ to $+\infty$ since infinitely many critical points of index $\lambda$ are added in addition to the already existing $(1+\alpha)$-many. It follows that $\Lambda M\cong \Lambda M'$ under homotopy. Likewise, $\pi_k(\Omega M)\cong \pi_k(\Omega M')$.

We will now use this homotopy equivalence $\pi_k(\Lambda M)\cong\pi_k(\Lambda M')$ to show that the homotopy type of $M$ is fixed. From the canonical isomorphisms, $\pi_k(\Lambda M)\cong \pi_k(\Omega M)\rtimes\pi_k(M)\cong \pi_k(\Lambda M')\cong \pi_k(\Omega M')\rtimes\pi_k(M')$. By $\pi_k(\Omega M)\cong \pi_k(\Omega M')$, it follows that $\pi_k(M)\cong\pi_k(M')$. Let $f:M\to M'$, then, because $\pi_k(M,x_0)\cong\pi_k(M',f(x_0))$ for $x_0\in M$, it follows that $f$ is a homotopy equivalence so $M$ and $M'$ are of the same homotopy type. This is similarly seen through the isomorphism $\pi_k(\Omega M)\cong\pi_{k-1}(M)$.
\end{proof} 

Let $Y_k=Y_{k-1}\bigcup_{\alpha_k\in A_k}^{\theta_k} H^{\lambda_k}$ be a manifold $Y_{k-1}$ with $|A_k|$-many $\lambda_k$-handles $H^{\lambda_k}=D^{\dim Y_{k-1}-\lambda_k}\times D^{\lambda_k}$ attached along the embedding map $\theta_k:S^{\lambda_k-1}\times D^{\dim Y_{k-1}-\gamma_k}\to\partial Y_{k-1}$ for $1\le k\le N$ and $1\le \alpha_k\le |A_k|$ with $Y_0:=D^{\infty}$ a Baire space and $Y:=Y_N$. Thus, free loop space has a handle decomposition given by $\Lambda M:=Y=D^{\infty}\bigcup_{k=1}^N\left[\bigcup_{\alpha_k\in{A_k}}^{\theta_k}H^{\lambda_k}\right]$ where $\lambda_k$ is the index of $|A_k|$-many critical points of the Morse functional $E:\Lambda M\to\mathbb{R}.$

There is a deformation retraction of $Y_k=Y_{k-1}\bigcup_{\alpha_k\in A_k}^{\theta_k} H^{\lambda_k}$ onto the core $Y_{k-1}\bigcup_{\alpha_k\in A_k}^{\theta_k}(D^{\lambda_k}\times \{0\})$ so homotopically attaching a $\lambda_k$-handle is the same as attaching a $\lambda_k$-cell. Thus, $Y_k\approx Y_{k-1}\bigcup_{\alpha_k\in A_k}^{\theta_k} (D^{\lambda_k}\times \{0\})\cong Y_{k-1}\bigcup_{\alpha_k\in A_k}^{\theta_k}e^{\lambda_k}$ for $e^{\lambda_k}$ a $\lambda_k$-cell. In the handle decomposition, we collapse each handle $D^m\times D^{n-m}$ to obtain a homotopy equivalent CW complex $X\approx_G\Lambda M$ under homotopy $G:X\to Y$. Thus, let $n_{\lambda_i}$ denote the number of $\lambda_i$-handles of $\Lambda M$. Then $\Lambda M$ has an infinite-dimensional CW structure with $n_{\lambda_i}$-many $\lambda_i$-cells for each $1\le i\le N$.

In particular, we construct the CW complex for $\Lambda M$ as follows:

\begin{enumerate}[label=(\roman*)]
\item Set $X_0:=D^{\infty}$, which is regarded as an $\infty$-cell $D^\infty\times \{0\}$ modulo a point $\{0\}$.
\item Form the $k$-skeleton $X_k$ from $X_{k-1}$ by attaching $\lambda_k$-cells $e_{\alpha}^{\lambda_k}$ via attaching maps $\phi_{\alpha}^k:S^{k-1}\to X_{k-1}$ for $e_{\alpha}^{\lambda_k}$. Here, we abbreviate $X_{\lambda_k}$ by $X_k$.
\item We stop the inductive process with $\Lambda M\approx X=X_N$ for $N< \infty$.
\end{enumerate} 
\begin{definition}
Every cell $e_{\alpha}^k$ in the cell complex $X$ has characteristic map $\Phi_{\alpha}^k: D_{\alpha}^k\to X$ extending the attaching map $\phi_{\alpha}^k$. It is a homeomorphism from the interior of $D_{\alpha}^k$ onto $e_{\alpha}^k$. Namely, $\Phi_{\alpha}^k$ is the composition $D_{\alpha}^k\hookrightarrow X_{k-1}\sqcup_{\alpha} D_{\alpha}^k\to X_k\hookrightarrow X$ where the middle map is a quotient map defining $X_k$. 
\end{definition}
The CW complex $X$ is infinite-dimensional because the $0$-th skeleton $X^0=D^{\infty}$ is infinite-dimensional. We topologize $X$ by saying a set $A\subset X$ is open (resp. closed) if and only if $A\cap X_k$ is open (resp. closed) in $X_k$ for each $k$. That is, a set $A\subset X$ is open (resp. closed) if and only if $\Phi_{\alpha}^{-1}(A)$ is open (resp. closed) in $D_{\alpha}^k$ for each characteristic map $\Phi_{\alpha}$. The first direction follows from continuity of the characteristic maps $\Phi_{\alpha}$. In the other direction, suppose $\Phi_{\alpha}^{-1}(A)$ is open in $D_{\alpha}^k$ for each $\Phi_{\alpha}$. Further suppose by induction on $k$ that $A\cap X_{k-1}$ is open in $X_{k-1}$. Since $\Phi_{\alpha}^{-1}(A)$ is open in $D_{\alpha}^k$ for all $\alpha$, $A\cap X_k$ is open in $X_k$ by definition of the quotient topology on $X_k$. Thus $A$ is open in $X$, from which it follows that $X$ is the quotient space of $\bigsqcup_{k,\alpha}D_{\alpha}^k$. The $(k+1)$-th skeleton is obtain from the $k$-skeleton by attaching $(k+1)$-cells; that is, there is a pushout diagram:
\begin{equation*}
\begin{tikzcd}[column sep=large, row sep=large]
\bigsqcup S^k \arrow{d} \arrow{r}
& X_k \arrow{d} \\
\bigsqcup D^{k+1} \arrow{r}
& X_{k+1}
\end{tikzcd}
\end{equation*}
where the disjoint unions are over all $(k+1)$-cells of $X$.

\begin{lemma}\label{CWcomplex}
Since $X$ is a CW complex, it follows that:
\begin{enumerate}[label=(\roman*)]
\item $H_k(X_n,X_{n-1})=\begin{cases}
\mathbb{Z}^{\#\lambda_n-\text{cells}}, & \text{if } k =n\\
0, & \text{if } k\ne n.
\end{cases}$ \\ That is, $H_k(X_n,X_{n-1})$ is free abelian for $k=n$ with a basis in one-to-one correspondence with the $n$-cells of $X$.
\item $H_k(X_n)=0$ if $k>n$. In particular, if $X$ is finite-dimensional, then $H_k(X)=0$ if $k>\dim (X)$. Note, we exclude the latter condition because $X$ under present consideration is, indeed, infinite-dimensional.
\item The map $H_k(X_n)\to H_k(X)$ induced by the inclusion $i:X_n\hookrightarrow X$ is an isomorphism for $k<n$ and a surjection for $k=n$.
\end{enumerate} 
\end{lemma}\label{cellularhomology}
\begin{proof}
(i) Consider the CW pair $(X_n, X_{n-1})$. The skeleton $X_n$ is obtained from $X_{n-1}$ by attaching the $\lambda_n$-cells $(e_{\alpha}^{\lambda_n})_{\alpha}$. Choose a point $x_{\alpha}$ at the center of each $\lambda_n$-cell $e_{\alpha}^{\lambda_n}$. Furthermore, let $A=X_n-\{x_{\alpha}\}_{\alpha}$. Thus, $A$ deformation retracts to $X_{n-1}$ so that $H_k(X_n,X_{n-1})\cong H_k(X_n,A)$. It follows that $H_k(X_n,A)\cong \bigoplus_{\alpha} H_k(D_{\alpha}^n, D_{\alpha}^n-\{x_{\alpha}\})$ by excision of $X_{n-1}$. From the long exact sequence for the pair $(D_{\alpha}^n,D_{\alpha}^n-\{x_{\alpha}\})$, we obtain
\begin{equation*}
H_k(D_{\alpha}^n, D_{\alpha}^n-\{x_{\alpha}\})\cong \tilde{H}_{k-1}(S_{\alpha}^{n-1})\cong\begin{cases}
\mathbb{Z}, & \text{if }k=n\\
0, & \text{if }k\ne n.
\end{cases}
\end{equation*} The first assertion follows at once.\\
(ii) Consider the long exact sequence for the CW pair $(X_n, X_{n-1})$:
\begin{equation*}
\cdots\to H_{k+1} (X_n, X_{n-1})\to H_k(X_{n-1})\to H_k(X_n)\to H_k(X_n,X_{n-1})\to\cdots
\end{equation*} Then if $k+1\ne n$ and $k \ne n$, from (i), it follows that $H_{k+1}(X_n,X_{n-1})=0$ and $H_k(X_n,X_{n-1})=0$. Therefore, $H_k(X_{n-1})\cong H_k(X_n)$. If $k>n$, i.e. for $n\ne k +1$ and $n\ne k$, we have that 
\begin{equation*}
H_k(X_n)\cong H_k(X_{n-1})\cong H_k(X_{n-2})\cong\dots\cong H_k(X_0).
\end{equation*} In the case under consideration, $X_0$ is a Baire space $D^{\infty}$ so $H_k(D^{\infty})=0$ if $k\ne 0$. Then for $k>n$, $H_k(X_n)=0$ as was to be shown. \\
(iii) We prove the statement for an infinite-dimensional CW complex $X$. Let $k<n$ and consider the long exact sequence for the CW pair $(X_{n+1},X_n)$ of sub-complexes:
\begin{equation*}
\cdots\to H_{k+1}(X_{n+1},X_n)\to H_k(X_n)\to H_k(X_{n+1})\to H_k(X_{n+1}, X_n)\to\cdots
\end{equation*} Recall that $k<n$ so $k+1\ne n+1$ and $k\ne n+1$ or, by (i), $H_{k+1}(X_{n+1},X_n)=0$ and $H_k(X_{n+1},X_n)=0$. Then from exactness of the above sequence, $H_k(X_n)\cong H_k(X_{n+1})$ and, by iteration, we obtain:
\begin{equation*}
H_k(X_n)\cong H_k(X_{n+1})\cong H_k(X_{n+2})\cong\dots\cong H_k(X_{n+l})=H_k(X),
\end{equation*} where $l$ is chosen such that $X_{n+l}=X$ which is admissible since the indices of the handles of $X$ are finite by definition and $X:=X_N=X_{N-1}\cup_{\theta} H^{\lambda_N}$ where $\lambda_i<\infty$ for $1\le i\le N$.
\end{proof}

\begin{definition}
The cellular homology $H_*^{CW}(X)$ of the CW complex $X$ under consideration. is defined as the homology of the cellular chain complex $(C_*(X),d_*)$ indexed by the cells of $X$. That is, the $n$-th abelian group in the chain complex is $C_n(X):=H_n(X_n,X_{n-1})=\mathbb{Z}^{\#\lambda_n-\text{cells}}$ with boundary operators $d_n:C_n(X)\to C_{n-1}(X)$ defined by the diagram:

\fontsize{9}{12}\selectfont{
\begin{tikzcd}
 & & & H_n(X_{n+1},X_n) = 0 & \\
H_n(X_{n-1}) = 0 \ar[rd] & & H_n(X_{n+1})\cong H_n(X) \ar[ru] & & \\
 & H_n(X_n) \ar[ru,"i_n"] \ar[rd,"j_n"] & & & \\
H_n(X_{n+1},X_n) = 0 \ar[ru,"\partial_{n+1}"] \ar[rr,"d_{n+1}"] & & H_n(X_n,X_{n-1}) \ar[dr,"\partial_n"'] \ar[rr,"d_n"] & & H_{n-1}(X_{n-1},X_{n-2}) \\
 & & & H_{n-1}(X_{n-1}) \ar[ru,"j_{n-1}"'] & \\
 & & H_{n-1}(X_{n-2}) = 0 \ar[ru] & & \\
\end{tikzcd}}
\end{definition}

The diagonal arrows are due to the long exact sequences of relative homology, and we invoke Lemma \ref{cellularhomology} to conclude $H_n(X_{n-1})=0$, $H_{n-1}(X_{n-2})=0$ and $H_n(X_{n+1})\cong H_n(X)$ in the above diagram. Then we define 
\begin{equation*} 
d_n=j_{n-1}\circ\partial_n:C_n(X)\to C_{n-1}(X)
\end{equation*} and observe that $d_n\circ d_{n+1}=0$. Moreover, $d_n\circ d_{n+1}=j_{n-1}\circ \partial_n \circ j_n\circ \partial_{n+1}=0$ since the composition of two consecutive maps in a long exact sequence is the zero map $\partial_n\circ j_n=0$.

We henceforth use the notations of Section \ref{sec:rathomotopy}. For two real numbers $a,b$ where $b>a$, define the sets $\Lambda^{[a,b]}$ by $\Lambda^{[a,b]}=\{p\in\Lambda M: a\le E(p)\le b\}$. Recall, if $\Lambda^{[a,b]}$ is compact and contains no critical point of $E:\Lambda M\to\mathbb{R}$ then $\Lambda^a$ is diffeomorphic to $\Lambda^b$. On the other hand, suppose that $\Lambda^{[a,b]}$ is compact and contains exactly one non-degenerate critical point of $E$. If the index of the critical point is $\lambda$ then $\Lambda^b$ is obtained from $\Lambda^a$ by successively attaching a handle of index $\lambda$ and a collar.
\begin{proposition}
Every smooth Hilbert manifold admits a Morse function $E:\Lambda M\to\mathbb{R}$ such that $\Lambda^a$ is compact for every $a\in\mathbb{R}$.
\end{proposition}
\begin{theorem}\label{handledecomposition}
Let $E:\Lambda M\to\mathbb{R}$ be a Morse function on a paracompact manifold $\Lambda M$ such that each $\Lambda^a$ is paracompact. Let $p_1,p_2,\dots,p_k,\dots$ be the critical points of $E$ and let the index of $p_i$ be $\lambda_i$. Then $\Lambda M$ has a handle decomposition which admits a handle of index $\lambda_i$ for each $i$. 
\end{theorem}
\begin{proof}
Consider the Morse function $E:\Lambda M\to\mathbb{R}$. Let $c_1<c_2<\dots<c_k<\dots$ be the $m$ critical values of $E$. Then consider the sequence of real numbers $\{a_i\}$ for which $c_k<a_k<c_{k+1}$. For real $a,b$ with $b>a$, we let $\Lambda^{[a,b]}=\{p\in\Lambda M:a\le E[p]\le b\}$. Consequently, $\Lambda^{a_{k+1}}=\Lambda^{a_k}\cup\Lambda^{[a_k,a_{k+1}]}$ and $\Lambda M=\bigcup_{k=0}^{m-1}\Lambda^{[a_k,a_{k+1}]}$. The set $\Lambda^{[a_k,a_{k+1}]}$ is paracompact so it only contains finitely many critical points of $E$. If $\lambda_{k+1}$ is the index of $c_{k+1}$, it has the homotopy type of a handle of index $\lambda_{k+1}$.
\end{proof}
\begin{corollary}
Every smooth manifold has the homotopy type of a CW complex.
\end{corollary}
\begin{proof}
The set $\Lambda^b$ is homotopy equivalent to $\Lambda^a$ with a $1$-cell attached.
\end{proof}

\begin{definition}\label{Smithnormal}
Let $A\in\text{Mat}_{m,n}(\mathbb{Z})$ be a matrix over the principal ideal domain $\mathbb{Z}$. Then the Smith normal form is a factorization $A=UDV$ for which:
\begin{enumerate}[label=(\roman*)]
\item $D\in\text{Mat}_{m,n}(\mathbb{Z})$ is diagonal where $d_{i,j}=0$ if $i\ne j$.
\item Every diagonal entry of $D$ divides the next, i.e. $d_{i,i}|d_{i+1,i+1}$. Such diagonal entries are called the elementary divisors of $A$.
\item $U\in\text{Mat}_{m,m}(\mathbb{Z})$ and $V\in\text{Mat}_{n,n}(\mathbb{Z})$ are invertible over $\mathbb{Z}$. That is, $\det U,\det V=\pm 1$ for $U,V$ in $SL$(---,$\mathbb{Z})$.
\end{enumerate} 
\end{definition}
Similarly, for $A\in\text{Mat}_{m,n}(\mathbb{Z})$ an integer matrix, the Smith normal form $A=UDV$ may be written as $D=U^{-1}AV^{-1}$, which is permissible by the invertibility of $U$ and $V$. Here, the product $U^{-1}AV^{-1}$ is
\begin{equation}
SNF(A)=\begin{bmatrix}\alpha_1 & 0 & 0 & & \cdots & & 0\\0& \alpha_2 & 0 & & \cdots & & 0\\ 0 & 0 & \ddots & & & & 0\\ \vdots& & & \alpha_r & & & \vdots \\ & & & & 0 & & & \\ & & & & & \ddots & \\ 0 & & & \cdots & & & 0 \end{bmatrix}
\end{equation} where the diagonal elements $\alpha_i$, unique up to multiplication by a unit of the principal ideal domain, satisfy $\alpha_i|\alpha_{i+1}$ for all $1\le i <r$. The elements $\alpha_i$ are the elementary divisors of $A$, determined by
\begin{equation}
\alpha_i=\frac{d_i(A)}{d_{i-1}(A)}
\end{equation} where $d_i(A)$, known as the $i$-\textit{th determinant divisor}, is the greatest common divisor of all $i\times i$ matrix of minors of $A$, and $d_0(A):=1$.
\begin{remark}
The Smith normal form exists for matrices over arbitrary principal ideal domains $R$ whereby $U$ and $V$ must be invertible over $R$ so that their determinants are any units in $R$. The Smith form is unique up to multiplication by units in $R$.
\end{remark}
\begin{lemma}[Omar Camarena]\label{cokernel}
Let the non-zero elementary divisors of an $m\times n$ integer matrix $A$ be $\alpha_1,\dots,\alpha_r$ where $r:=\text{rank }(A)$. Then $\text{coker}(A):=\mathbb{Z}^m/\text{im}(A)\cong\bigoplus_{i=1}^r\mathbb{Z}/a_i\oplus\mathbb{Z}^{m-r}$.
\end{lemma}
\begin{proof}
Consider the Smith normal form $UDV$ of $A$. Then $\text{coker}(A)\cong\text{coker}(D)$ under the map $x+\text{im}(A)\mapsto U^{-1}x+\text{im}(D)$. To compute $\text{coker}(D)$, if $e_1,\dots,e_m$ is the standard basis of $\mathbb{Z}^m$, then the image of $D$ is spanned by $a_1e_1,\dots,a_re_r$.
\end{proof}
\begin{theorem}[Omar Camarena]\label{Smithhomology}
The homology of a complex $\mathbb{Z}^n\overset{A}\longrightarrow\mathbb{Z}^m\overset{B}\longrightarrow\mathbb{Z}^k$ (with $BA=0$) is $\ker(B)/\text{im}(A)\cong\bigoplus_{i=1}^r\mathbb{Z}/{\alpha}_i\oplus\mathbb{Z}^{m-r-s}$ where $r:=\text{rank(B)}$ and $\alpha_1,\dots,\alpha_r$ are the non-zero elementary divisors of $A$.
\end{theorem}
\begin{proof}
The boundary map $B$ is zero on $im(A)$ so it descends to a homomorphism $\bar{B}:\mathbb{Z}^m/\text{im}(A)\to\mathbb{Z}^k$. We will now show that $\ker(B)/\text{im}(A)$ is isomorphic to $\ker(\bar{B})$. From Lemma \ref{cokernel}, $\mathbb{Z}^m/\text{im}(A)\cong\bigoplus_{i=1}^r\mathbb{Z}/\alpha_i\oplus\mathbb{Z}^{m-r}$. The target $\mathbb{Z}^k$ of $B$ is free so that the kernel contains all of the torsion of its domain group. It follows that
\begin{equation*}
\ker(\bar{B})=\bigoplus_{i=1}^r\mathbb{Z}/\alpha_i\oplus\ker\left(\bar{B}|_{\mathbb{Z}^{m-r}}\right).
\end{equation*} The kernel of the restriction of $\bar{B}$ to its free part is also free and, thus, the rank is 
\begin{equation}
\text{rank }\ker(\bar{B})=m-r-s
\end{equation} for $s$ the rank of $\text{im}(B)=\text{im}(\bar{B})=\text{im}\left(\bar{B}|_{\mathbb{Z}^{m-r}}\right)$.
\end{proof}

\begin{theorem} [Main Result]\label{main result}
Every closed Riemannian manifold $M$ admits infinitely many geometrically distinct, non-constant, prime closed geodesics.
\end{theorem}
\begin{proof}
Let $X\approx_{G} \Lambda M$ be the homotopy equivalence defined in the construction preceding Lemma \ref{CWcomplex}. Consider a section of the cellular chain complex for the CW pair $(X_k,X_{k-1})$:
\begin{equation}\label{eq:complex}
\cdots\longrightarrow H_{k+1}(X_{k+1}, X_k)\overset{M_{k+1}}\longrightarrow H_k(X_k, X_{k-1})\overset{M_k}\longrightarrow H_{k-1}(X_{k-1},X_{k-2})\longrightarrow\cdots \cdots
\end{equation}
or equivalently, by the definition of $C_k(X):=H_k(X_k,X_{k-1})=\mathbb{Z}^{\#\lambda_k-\text{cells}}$,
\begin{equation*}
\cdots\longrightarrow\mathbb{Z}^{\#\lambda_{k+1}-\text{cells}}\overset{M_{k+1}}\longrightarrow \mathbb{Z}^{\#\lambda_k-\text{cells}}\overset{M_k}\longrightarrow \mathbb{Z}^{\#\lambda_{k-1}-\text{cells}}\longrightarrow \cdots
\end{equation*} where
\begin{equation}
M_k:  H_k(X_k, X_{k-1})\to H_{k-1}(X_{k-1}, X_{k-2})
\end{equation}
is a homomorphism from one free abelian group, indexed by $\lambda_k$-cells attached to go from $X_{k-1}$ to $X_k$, to another. Any such homomorphism is given by an integer-valued matrix $M_k$. In what follows, we determine the $(i,j)$-entries of the matrix
\begin{equation*}
M_k:  H_k(X_k, X_{k-1})\to H_{k-1}(X_{k-1}, X_{k-2})
\end{equation*} which is given by a particular CW presentation of the CW complex $X$. For each entry of this matrix, take the $i$-th $\lambda_k$-cell represented by a continuous map $f_i: S^{k-1}\to X_{k-1}$, and take the $j$-th $\lambda_{k-1}$-cell represented by a map $g_j: (D^{k-1}, S^{k-2})\to (X_{k-1}, X_{k-2})$. Then choose any point $y$ in the interior of $D^{k-1}$, seen as a point $g_j(y)$ in $X_{k-1}$, and look at the points $x_1,\dots, x_r$ in the inverse image $f^{-1}_i(y)$. Then we observe whether the restriction of $f_i$ to a small neighborhood of $x_t$ (for $1\le t\le r$), mapping to a small neighborhood of $y$, preserves or reverses the orientation. We put $+1$ if it preserves orientation and $-1$ if it reverses orientation; the sum of these signs is the corresponding entry of the matrix.

More precisely, we should take an \textit{interior} point $y$ in $D^{k-1}$; a small open neighborhood $V$ of $y$ maps homeomorphically onto its image in $X_{k-1}$ under the map $g_j: (D^{k-1}, S^{k-2})\to (X_{k-1}, X_{k-2})$  determined by the particular attaching map $S^{k-2}\to X_{k-2}$. This is a homeomorphism so the restriction of the map $f_i$, which takes $U = f_i^{-1}(g_j(V))$, an open neighborhood of $S^{k-1}$, to $g_j(V)$, is a map $U\to V$ between open subsets of manifolds. This applies in the case of a general CW complex, which might not itself be a manifold. For any continuous map $f: U\to V$ between manifolds, there is a \textit{continuously differentiable} map within the same homotopy class. Thus, we exclude the possibility that the attaching maps are pathological in nature.

Assuming $f_i$ is continuously differentiable, almost all points $y$ in $D^{k-1}$ are regular values of $f_i$ in the technical sense of measure theory. This follows by Sard's theorem which says that the set of critical values has measure zero. For a point $y$ to be \textit{regular} means that for each point $x$ in $f^{-1}_i(y)$, working in chosen coordinate patches, the derivative $Df_i$ at $x$, represented by a Jacobian matrix whose entries are partial derivatives at $x$ as a linear map from the tangent space at $x$ to the tangent space at $y$, has maximal rank. In the case where $U$ and $V$ have the same dimension, maximal rank means that $Df_i$ is invertible at $x$. Then, by the implicit function theorem, in the case where $y$ is regular, this inverse image $f^{-1}_i(y)$ in $S^{k-1}$ will be a compact $0$-dimensional manifold, and will hence consist of a finite number of points $x_1, \dots, x_r.$ 

The manifolds $S^{k-1}$ and $D^{k-1}$ are orientable. We may select standard Euclidean coordinates $(y_1, \dots, y_{k-1})$ for $y$ in $D^{k-1}$, which are the same coordinates for the open neighborhood $V$. As for each point $x$ in the inverse image of $y$, we can take any coordinate patch consistent with the standard orientation of $S^{k-1}$. Then, computing the derivative $Df_i$ at $x$ in terms of the chosen coordinatizations at $x$ and $y$ as a Jacobian matrix $\text{Jac}_x(f_i)$, the sign of the determinant of $\text{Jac}_x(f_i)$ is well-determined. The sign is either $+1$ or $-1$ since $Df_i$ is invertible. The \textit{degree} of $f_i$ is then computed as 
\begin{equation*}
\sum_{x \in f^{-1}_i(y)}  \text{sgn } \det \text{Jac}_x(f_i).
\end{equation*}
It is well-defined and independent of which regular value $y$ we pick. It is also invariant with respect to the homotopy class of the attaching map.

To write down the $(i,j)$-entry of the matrix $M_k$, we order the attaching maps that we used to get from $X_{k-1}$ to $X_k$:
\begin{equation*}
f_1: S^{k-1} \to X_{k-1},  f_2: S^{k-1}\to X_{k-1},\dots, f_{m_k}: S^{k-1} \to X_{k-1}
\end{equation*} where $m_k:=n_{\lambda_{k-1}}$ is the number of $\lambda_{k-1}$-cells of $X$. We assume there are only finitely many attaching maps for each skeletal dimension, although we can also drop the finiteness assumption. Likewise, we order the attaching maps one dimension down:
\begin{equation*}
g_1: S^{k-2}\to X_{k-2}, g_2: S^{k-2}\to X_{k-2},\dots, g_{n_k}: S^{k-2}\to X_{k-2}
\end{equation*} where $n_k:=n_{\lambda_k}$ is the number of $\lambda_k$-cells of $X$.
We think of the latter attaching maps in terms of giving maps of \textit{pairs} of spaces  $(D^{k-1}, S^{k-2})\to (X_{k-1}, X_{k-2})$, where the first map component  $D^{k-1}\to X_{k-1}$  identifies interior points of $D^{k-1}$ with points of $X_{k-1}$, and the second component is the actual attaching map $g_j: S^{k-2}\to X_{k-2}$. Then the matrix entry for $M_k\in\text{Mat}_{m_k,n_k}(\mathbb{Z})$ is 
\begin{equation*}
(M_k)_{ij}:=\deg(f_i)_j=\sum_{f_i(x) = y}  \text{sgn } \det \text{Jac}_x(f_i)
\end{equation*} for any choice of regular value $y$ of $f_i$ in the interior of the $j$-th cell attachment  $(D^{k-1}, S^{k-2})\to (X_{k-1}, X_{k-2})$ and for any choice of coordinatizations of small coordinate patches for the points $y$ and $x$ in the inverse image $f_i^{-1}(y)$ chosen compatibly with the orientations.  For $V$ an open neighborhood of $y$ in $D^{k-1}$, $f_i$ is a map between open neighborhoods of manifolds:
\begin{equation*}
\begin{split}
f_i:&S^{k-1}\longrightarrow X_{k-1},\\
& U=f_i^{-1}(g_j(V))\mapsto g_j(V)
\end{split}
\end{equation*} from which we may write the $(i,j)$-entry of $M_k$ as
\begin{equation}
(M_k)_{ij}:=\sum_{x\in f_i^{-1}(g_j(V))}\text{sgn }\det \text{Jac}_x(f_i).
\end{equation} Thus, 
\begin{equation}
M_k=\begin{bmatrix} \deg(f_1)_1 & \deg(f_1)_2 & \dots & \deg(f_1)_{n_{\lambda_k}}\\ \deg(f_2)_1& \deg(f_2)_2 & \dots & \deg(f_2)_{n_{\lambda_k}}\\ \vdots &\vdots & \ddots & \vdots \\ \deg\left(f_{n_{\lambda_{k-1}}}\right)_1 & \deg\left(f_{n_{\lambda_{k-1}}}\right)_2 & \dots & \deg\left(f_{n_{\lambda_{k-1}}}\right)_{n_{\lambda_k}}\end{bmatrix}.
\end{equation}
The homology of the cellular complex does not depend on the order chosen for the attaching maps. Since we have a cellular chain complex, the composite of any two matrices $M_{k+1}$ and $M_k$ is 0, and one wants to compute $\ker M_k / \text{im }M_{k+1}$  to get the $k$-th homology group of the CW complex $\Lambda M\approx_G X$.

We henceforth compute cellular homology for $X$ using the standard basis of $\mathbb{Z}^k$ for the free modules $C_k(X)$. To do so, we invoke Smith normal form of the matrix representation for the boundary homomorphisms. It should be remarked that if we want to determine homology with a basis over a field $K$, we would simply use Gaussian elimination. Note, $\mathbb{Z}$-modules are direct sums $\mathbb{Z}/\alpha_1\oplus\mathbb{Z}/\alpha_2\oplus\cdots\oplus\mathbb{Z}/\alpha_{l}$ for integers $\alpha_i$ with $\alpha_i | \alpha_{i+1}$. We reduce an integer matrix $M_k$ to $SNF(M_k)$ using row and column operations as in standard Gaussian elimination; but, the caveat is that the entries must remain integers at all stages. Recall $M_k$ is $m_k\times n_k$ where $m_k=n_{\lambda_{k-1}}$ and $n_k=n_{\lambda_k}$. The row and column operations correspond to a change of basis for $M_k$ and $M_{k+1}$, respectively. Thus, the Smith normal form gives a factorization $M_k=U_kD_kV_k$ with
\begin{equation}
\begin{split}
D_k=\begin{bmatrix} B_k & \bm{0} \\ \bm{0} &\bm{0} \end{bmatrix}, \text{ where }
B_k=\begin{bmatrix} 
\alpha_1 & & & & 0 \\ 
 & & \ddots & &  \\
 & & & & \\
 0 & & & & \alpha_{l_k}
\end{bmatrix}
\end{split}
\end{equation} for $B_k$ a matrix with $l_k$ non-zero diagonal entries which satisfy $\alpha_i\ge 1$. The SNF matrices for $M_{k+1}$ and $M_k$ completely characterize the $k$-th homology group $H_k$. The rank of the boundary group $\text{im }M_{k+1}$ is the number of non-zero rows of $D_{k+1}$, i.e. $l_{k+1}=\text{rank }\text{im }M_{k+1}=\min(n_{\lambda_k},n_{\lambda_{k+1}})$. The rank of the cycle group $\ker M_k$ is the number of zero columns of $D_k$, i.e. $n_k-l_k$. The torsion coefficients of $H_k$ are the diagonal entries $\alpha_i$ of $M_{k+1}$ that are greater than one. The $k$-th Betti number of the CW complex $\Lambda M$ is thus:
\begin{equation}
b_k(\Lambda M;\mathbb{Z})=\text{rank }\ker M_k-\text{rank }\text{im } M_{k+1}=n_k-l_k-l_{k+1}.
\end{equation}

Similarly, the elementary divisors are given by $\alpha_i=d_i(M_k)/d_{i-1}(M_k)$ where $d_i(M_k)$ is the $i$-determinant divisor for $1\le i \le l_k$. In particular, $d_i(M_k)$ is the greatest common divisor of all $i\times i$ minors of $M_k$. Denote by $\Sigma_{i,k}$ the $i\times i$ sub-matrix of $M_k$. Then an $i\times i$ minor of $M_k$, i.e. the minor determinant of order $i$, is the determinant of an $i\times i$ matrix obtained from $M_k$ by deleting $m_k-i$ rows and $n_k-i$ columns. The minor of order zero is defined to be $+1$. For the matrix $M_k$, there are ${m_k\choose i}{n_k\choose i}$ minors of size $i\times i$. One such minor $|\Sigma_{i,k}|$ is the determinant of the $i$-th sub-matrix $\Sigma_{i,k}$:
\begin{equation} |\Sigma_{i,k}|=
\begin{vmatrix}
(M_k)_{\sigma,\rho} & \dots & & (M_k)_{\sigma,\rho+i} & \\ 
\hspace*{4mm} (M_k)_{\sigma+1,\rho} & \dots & & \hspace*{3.5mm}(M_k)_{\sigma+1,\rho+i} & \\
\vdots & \ddots & & \vdots \\
\hspace*{2mm}(M_k)_{\sigma+i,\rho} & \dots & & \hspace*{3.4mm} (M_k)_{\sigma+i,\rho+i} &
\end{vmatrix}.
\end{equation} for $1\le\sigma\le m_k$ and $1\le\rho\le n_k$. There are $i!$ orderings of the columns $(\rho,\rho+1,\dots,\rho+i)$ that go in each possible order $(\alpha,\beta,\dots,\omega)$. We choose $(M_k)_{\sigma,\alpha}$ from row $1$ of $\Sigma_{i,k}$ (row $\sigma$ of $M_k$), $(M_k)_{\sigma+1,\beta}$ from row $2$ of $\Sigma_{i,k}$ (row $\sigma+1$ of $M_k$), and eventually $(M_k)_{\sigma+i,\omega}$ from row $i$ of $\Sigma_{i,k}$ (row $\sigma+i$ of $M_k$). Then the determinant contains the product $(M_k)_{\rho,\alpha}(M_k)_{\rho+1,\beta}\dots (M_k)_{\rho+i,\omega}$ times $+1$ or $-1$. Let $P=(\alpha,\beta,\dots,\omega)$ be the permutation matrices corresponding to a certain ordering of the columns of $\Sigma_{i,k}$. Then the determinant of $\Sigma_{i,k}$ is the sum of the $i!$ simple determinants (up to sign), which are given by choosing one entry from every row and column of $\Sigma_{i,k}$. Therefore,
\begin{equation}
|\Sigma_{i,k}|=\sum_{P=(\alpha,\beta,\dots,\omega)}(\det P)(M_k)_{\sigma,\alpha}(M_k)_{\sigma+1,\beta}\dots (M_k)_{\sigma+i,\omega}
\end{equation} for $\rho\le \alpha,\beta,\dots,\omega\le \rho+i$. Then the $i$-th determinant divisor of $M_k$ is the greatest common divisor of all $i\times i$ minors $|\Sigma_{i,k}|$. Let $|\Sigma_{i,k}^r|$ be the $r$-th possible $i\times i$ minor of order $i$ of $M_k$ for $1\le r\le {m_k\choose i}{n_k\choose i}$. Then the $i$-th determinant divisor is
\begin{equation}
d_i(M_k)=\text{gcd}\left(\big |\Sigma_{i,k}^1\big |,\dots,\biggl |\Sigma_{i,k}^{{m_k\choose i} {n_k\choose i}}\biggr|\right).
\end{equation} Recall that for $a,b\in\mathbb{Z}$, $\text{gcd}(a,b)=\sum_{\substack{q|a, \\ q|b}}\varphi(q)$ where $\varphi(\cdot)$ is the Euler totient function. It follows that if we let $\Phi_j:=\text{gcd}(\Phi_{j-1},|\Sigma_{i,k}^j|)$ for $2\le j\le {m_k\choose i}{n_k\choose i}$ and $\Phi_1:=|\Sigma_{i,k}^1|$, then
\begin{equation}
\Phi_j=\sum_{\substack{q_{j-1}|\Phi_{j-1} \\ q_{j-1}||\Sigma_{i,k}^j|}}\varphi(q_{j-1})=\sum_{\substack{q_{j-1}|\Phi_{j-1} \\ q_{j-1}||\Sigma_{i,k}^j|}}q_{j-1}\prod_{p|q_{j-1}}\left(1-\frac{1}{p}\right)
\end{equation} for $p$ prime and $q_j|\Phi_j$ for all $j$. Moreover, we let $\tau(i,k):={m_k\choose i} {n_k\choose i}$ to obtain $d_i(M_k)=\Phi_{\tau(i,k)}$:
\begin{equation}
\begin{split}
d_i(M_k)&=\text{gcd}\left(|\Sigma_{i,k}^1|,\dots,|\Sigma_{i,k}^{\tau(i,k)}|\right)\\&=\sum_{\substack{q_{\tau(i,k)-1}|\Phi_{\tau(i,k)-1} \\ q_{\tau(i,k)-1} \big | \left|\Sigma_{i,k}^{\tau(i,k)}\right |}}q_{\tau(i,k)-1}\prod_{p|q_{\tau(i,k)-1}}\left(1-\frac{1}{p}\right).
\end{split}
\end{equation}

Recall that $X$ is the CW complex homotopy equivalent to $\Lambda M$. Applying Theorem \ref{Smithhomology} to the complex
\begin{equation}
\mathbb{Z}^{n_{\lambda_{k+1}}}\overset{M_{k+1}}\longrightarrow\mathbb{Z}^{n_{\lambda_k}}\overset{M_k}\longrightarrow\mathbb{Z}^{n_{\lambda_{k-1}}}
\end{equation} defined in equation (\ref{eq:complex}), we conclude that the $k$-th homology group of free loop space $\Lambda M$ is 
\begin{equation}
\begin{split}
H_k(\Lambda M;\mathbb{Z})&=\ker M_k/\text{im }M_{k+1}=\bigoplus_{i=1}^{l_{k+1}}\mathbb{Z}/\alpha_i\oplus\mathbb{Z}^{n_k-l_{k+1}-l_k}\\ &=\left. \bigoplus_{i=1}^{l_{k+1}}\mathbb{Z} \right/ \left(\frac{d_i(M_k)}{d_{i-1}(M_k)}\right)\oplus\mathbb{Z}^{n_k-l_{k+1}-l_k} \\ 
&=\left. \bigoplus_{i=1}^{l_{k+1}}\mathbb{Z} \right/ \left(\frac{\text{gcd}\biggl(\big |\Sigma_{i,k}^1\big |,\dots,\biggl |\Sigma_{i,k}^{{m_k\choose i} {n_k\choose i}}\biggr|\biggr)}{\text{gcd}\biggl(\big |\Sigma_{i-1,k}^1\big |,\dots,\biggl |\Sigma_{i-1,k}^{{m_k\choose i-1} {n_k\choose i-1}}\biggr|\biggr)}\right)\oplus\mathbb{Z}^{n_k-l_{k+1}-l_k}
\end{split}
\end{equation} where $l_{k+1}=\text{rank }\text{im }M_{k+1}=\min(n_{\lambda_k},n_{\lambda_{k+1}})$. Therefore, the $k$-th Betti number is
\begin{equation}
b_k(\Lambda M;\mathbb{Z})=n_k-l_k-l_{k+1}=n_{\lambda_k}-\min\left(n_{\lambda_{k-1}},n_{\lambda_k}\right)-\min\left(n_{\lambda_k},n_{\lambda_{k+1}}\right). 
\end{equation}

We use a $C^0$-perturbation of $E:\Lambda M\to\mathbb{R}$. That is, let $\sigma(t):=\max(0,\min(2t-1,1))$ for $t\in\mathbb{R}$. Then for $\varepsilon>0$, define 
\begin{equation}
E_{\varepsilon}(x):=E(\sigma(\|x\|/\varepsilon)x)=\int_{S^1}\|\dot{\sigma}(\|\gamma\|/\varepsilon)\gamma(t)+\sigma(\|\gamma\|/\varepsilon)\dot{\gamma}(t)\|^2dt.
\end{equation}
Therefore, $E_{\varepsilon}$ is constant in the ball of radius $\varepsilon/2$, it coincides with $E$ outside the ball of radius $\varepsilon$, and $\|E-E_{\varepsilon}\|_{\infty}=o(1)$ for $\varepsilon\to 0$. The homotopy types of $\Lambda M$ and $M$ are fixed under such a perturbation by Lemma \ref{invariance}, i.e. for $f:M\to M'$ and $x_0\in M$, $\pi_k(M,x_0)\cong\pi_k(M',f(x_0))$ for all $k$. Assume that we perturb $\Lambda M$ via $E_{\varepsilon}$ such that we create infinitely many minima or maxima in a local neighborhood of a critical point $\gamma\in X_k\subset \Lambda M$ of index $\lambda_k$. Then, by Lemma \ref{invariance}, the homotopy type of $M$ is fixed and the number of critical points of index $\lambda_k$ tends to infinity, i.e. $n_{\lambda_k}\to\infty$. Then the $k$-th Betti number of $H_k(\Lambda M;\mathbb{Z})$ becomes 
\begin{equation*}
\lim_{n_{\lambda_k}\to\infty}n_{\lambda_k}-\min\left(n_{\lambda_{k-1}},n_{\lambda_k}\right)-\min\left(n_{\lambda_k},n_{\lambda_{k+1}}\right)=+\infty
\end{equation*} where it is assumed that $n_{\lambda_i}$, $1\le i\le N$, is finite for all $i\ne k$. It follows that the sequence of Betti numbers $\{b_k(\Lambda M;\mathbb{Z})\}_{k\ge 0}$ is unbounded. Since the homotopy type of $M$ is invariant under the $C^0$-perturbation of the energy functional, this can be done for an arbitrary closed Riemannian manifold. Thus, every closed Riemannian manifold $M$ carries infinitely many geometrically distinct, non-constant, prime closed geodesics. 

If we instead compute cellular homology over a field $K$, say $\mathbb{Q}$, then we may simply use Gaussian elimination to reach the factorization $M_k=U_kD_kV_k$ for $U_k\in\text{Mat}_{m_k,m_k}(\mathbb{Q})$ and $V_k\in\text{Mat}_{n_k,n_k}(\mathbb{Q})$ where $U_kD_k$ is the inverse of the Gaussian elimination matrix $E_k$ which transforms $M_k$ to $V_k$. Note, each diagonal entry of $U_k$ is $+1$ since $D_k$ is the diagonal matrix that contains all the pivots of $M_k$. Then the rank of $\text{im }M_{k+1}$ is the number of non-zero rows of $D_{k+1}$, i.e. $l_{k+1}$, and the rank of $\ker M_k$ is the number of zero columns of $D_k$, i.e. $n_k-l_k$. Then the $k$-th homology group of $\Lambda M$ is
\begin{equation}
H_k(\Lambda M;\mathbb{Q})=\ker M_k/\text{im }M_{k+1}=\mathbb{Q}^{n_k-l_{k+1}-l_k}
\end{equation} and the $k$-th Betti number is given by $b_k(\Lambda M;\mathbb{Q})=n_k-l_k-l_{k+1}$. That is, all torsion for homology vanishes.
\end{proof} 
\begin{remark}
For a finite CW complex $X$, the Euler characteristic $\chi(X)$ is defined to be the alternating sum $\sum_k(-1)^kC_k$ where $C_k$ is the number of $k$-cells of $X$. The Euler characteristic of $X$ is independent of the choice of CW structure on $X$. We similarly define the Euler characteristic of  an infinite-dimensional CW complex $\Lambda M$ to be the alternating sum $\sum_k(-1)^kC_k$ where $C_k$ is the number of $k$-cells of $\Lambda M$ if and only if all homology groups have finite rank and all but finitely many are zero. In this case, $\Lambda M\approx_GX_N$ is such that $H_k(X_N)=0$ if $k>N$ by Lemma \ref{CWcomplex} so that all but finitely many homology groups are zero because the index of any critical point is finite. 
\end{remark}
\begin{definition}
The Poincar\'e polynomial of $\Lambda M$ is given by 
\begin{equation*}
P_{\Lambda M}(t)=\sum_{k=0}^{\infty}(\text{rank }H_k(\Lambda M;\mathbb{Z}))t^k.
\end{equation*} If the Euler characteristic is well-defined, i.e. if $\text{rank }H_k(\Lambda M;\mathbb{Z})=0$ for $k\ge N$ and a fixed $N$, then $\chi(\Lambda M)=P_{\Lambda M}(-1)$.
\end{definition}
Since $H_k(\Lambda M;\mathbb{Z})\cong H_k(X_N;\mathbb{Z})=0$ for $k>N$, we have:
\begin{equation} 
\begin{split}
\chi(\Lambda M)& =\sum_{k=0}^{\infty}(-1)^k\Big(n_{\lambda_k}-\min\left(n_{\lambda_{k-1}},n_{\lambda_k}\right)-\min\left(n_{\lambda_k},n_{\lambda_{k+1}}\right)\Big) \\ 
&=\sum_{k=0}^N(-1)^k\Big(n_{\lambda_k}-\min\left(n_{\lambda_{k-1}},n_{\lambda_k}\right)-\min\left(n_{\lambda_k},n_{\lambda_{k+1}}\right)\Big).
\end{split}
\end{equation} Homology $H_k$ is sub-additive due to the increasing sequence of inclusions $D^{\infty}=X_0\subset X_1\subset X_1\subset X_2\subset\dots\subset X_N\approx\Lambda M$ so we find
\begin{equation}
\text{rank }H_k(\Lambda M;\mathbb{Z})\le \sum_i\text{rank }H_k(X_i,X_{i-1};\mathbb{Z})=\text{rank }H_k(X_k,X_{k-1};\mathbb{Z})=n_{\lambda_k}
\end{equation} from 
\begin{equation*}
H_k(X_i,X_{i-1};\mathbb{Z})=\begin{cases}
\mathbb{Z}^{\#\lambda_i-\text{cells}}, &\text{if }k=i\\
0, & \text{if }k\ne i.
\end{cases}
\end{equation*} That is, the rank of $H_k(\Lambda M;\mathbb{Z})$ is less than or equal to the number of critical points of index $\lambda_k$. Equivalently, $\text{rank }H_k(\Lambda M;\mathbb{Z})=n_{\lambda_k}-l_k-l_{k+1}\le n_{\lambda_k}$. Since the $k$-th Betti number of the free loop space $\Lambda M$ is given by $b_k(\Lambda M;\mathbb{Z})=n_{\lambda_k}-\min\left(n_{\lambda_{k-1}},n_{\lambda_k}\right)-\min\left(n_{\lambda_k},n_{\lambda_{k+1}}\right)$, the least upper bound of the sequence of Betti numbers is $\sup_{k\in\mathbb{N}}b_k(\Lambda M;\mathbb{Z})=n_{\lambda_k}$. Thus, the supremum of the $k$-th Betti number is controlled by the number of handles of index $\lambda_k$.

\section{Acknowledgements}
I would like to sincerely thank Dr. John Pardon for his guidance and suggestions regarding the content of this paper. I also thank Dr. Todd Trimble for his many useful conversations on techniques employed in this analysis.

\newpage
\bibliographystyle{plain}

\begin{thebibliography}{9}
\bibitem{brown}
K. S. Brown. \textit{Cohomology of groups}, volume 87 of Graduate Texts in Mathematics. Springer-Verlag, New York, 1994. Corrected reprint of the 1982 original.
\bibitem{deRham}
G. De Rham, \textit{Sur la r\'eductibilit\'e d'un espace de Riemann}, Comm. Math. Helv., vol. 26
(1952), pp. 328-344.
\bibitem{felix-halperin}
Y. F\'elix; S. Halperin; J. Thomas (1993), \textit{Elliptic spaces \rom{2}, L'Enseignement Math\'ematique}, doi:10.5169/seals-60412
\bibitem{gromoll-meyer}
D. Gromoll and W. Meyer. \textit{Periodic geodesics on compact Riemannian manifolds}. J. Differential Geometry, 3:493--510, 1969.
\bibitem{helgason}
S. Helgason. \textit{Differential Geometry and Symmetric Spaces}. AMS Chelsea Publishing, 1962.
\bibitem{hingston}
N. Hingston. \textit{Curve shortening, equivariant Morse theory, and closed geodesics on the 2-sphere}. In Differential geometry: Riemannian geometry (Los Angeles, CA, 1990), volume 54 of Proc. Sympos. Pure Math., pages 423--429. Amer. Math. Soc., Providence, RI, 1993.
\bibitem{kato}
T. Kato. \textit{Perturbation theory for linear operators}. Classics in Mathematics. Springer-Verlag, Berlin, 1995. Reprint of the 1980 ed.
\bibitem{klingenberg}
W. Klingenberg. \textit{Lectures on closed geodesics}. Springer-Verlag, Berlin, 1978.
Grundlehren der Mathematischen Wissenschaften, Vol. 230.
\bibitem{lyusternik-fet}
L. A. Lyusternik and A. I. Fet. \textit{Variational problems on closed manifolds}. Doklady Akad. Nauk SSSR (N.S.), 81:17--18, 1951.
\bibitem{alexandru}
A. Oancea.\textit{Morse theory, closed geodesics, and the homology of free loop spaces}. arXiv:1406.3107, 2014.
\bibitem{palais-smale}
R. S. Palais and S. Smale. \textit{A generalized Morse theory}. Bull. Amer. Math. Soc., 70:165--172, 1964.
\bibitem{palais-hilbertmanifolds}
R. S. Palais. \textit{Morse theory on Hilbert manifolds}. Topology, 2:299--340, 1963.
\bibitem{poirrier-sullivan}
M. Vigu\'e-Poirrier and D. Sullivan. \textit{The homology theory of the closed geodesic problem}. J. Differential Geometry, 11(4):633--644, 1976.
\bibitem{rademacher}
H.-B. Rademacher. \textit{On the average indices of closed geodesics}. J. Differential Geom., 29(1):65--83, 1989.
\bibitem{sarnak}
P. Sarnak. \textit{Prime Geodesic Theorems}. Stanford University, 1980.
\bibitem{ziller}
W. Ziller. \textit{The free loop space of globally symmetric spaces}. Invent. Math., 41(1):1--22, 1977.
\end{thebibliography}

\end{document}